\documentclass{svproc}
\usepackage{empheq}
\usepackage[margin=1in]{geometry}         

\usepackage{cite}
\usepackage{empheq}
\usepackage{fullpage}
\usepackage{longtable}
\usepackage{pdfpages}
\usepackage{caption,subcaption}
\usepackage{stmaryrd}
\usepackage{framed}
\usepackage{amssymb}
\usepackage{epstopdf}
\usepackage{amscd}
\usepackage{amsmath,amsfonts,amssymb}
\usepackage{enumerate}
\usepackage{multirow}
\usepackage{subcaption}
\usepackage{csquotes}
\usepackage{fixltx2e}
\usepackage{bbm}
\usepackage{url}

\usepackage{cite}
\usepackage{makeidx}
\usepackage{fullpage}
\usepackage{fancyhdr}
\usepackage{empheq}
\usepackage{color}
\usepackage{float}
\usepackage{soul}
\usepackage[doc]{optional}
\usepackage{enumitem}
\usepackage[toc,page]{appendix}
\usepackage{theorem}
\usepackage{array}
\usepackage{booktabs}
\usepackage{epsfig}
\usepackage{latexsym}
\usepackage{subcaption}
\usepackage{adjustbox}
\usepackage[percent]{overpic}

\usepackage{tikz}
\usetikzlibrary{decorations.pathmorphing}
\usepackage{graphicx, xcolor, soul}

\definecolor{myblue}{rgb}{.9, .9, 1}

\DeclareMathOperator{\proj}{P}
\DeclareMathOperator{\dm}{d}

\newcommand{\lev}{{\rm lev}}

\newcommand{\A}{\mathcal{A}}
\newcommand{\B}{\mathcal{B}}

\theoremstyle{plain}{\theorembodyfont{\rmfamily}

\theoremstyle{plain}{\theorembodyfont{\rmfamily}
}
\theoremstyle{plain}{\theorembodyfont{\rmfamily}
}
\theoremstyle{plain}{\theorembodyfont{\rmfamily}
}
\theoremstyle{plain}{\theorembodyfont{\rmfamily}

\theoremstyle{plain}{\theorembodyfont{\rmfamily}

\theoremstyle{plain}{\theorembodyfont{\rmfamily}

\theoremstyle{plain}{\theorembodyfont{\rmfamily}

\theoremstyle{plain}{\theorembodyfont{\rmfamily}

\theoremstyle{plain}

\newcommand{\Fix}{\ensuremath{\operatorname{Fix}}}
\newcommand{\Id}{\ensuremath{\operatorname{Id}}}

\newcommand{\bdry}{\ensuremath{\operatorname{bd}}}

\def\RR{{\mathbb{R}}}

\def\NN{{\mathbb{N}}}

\newcommand{\argmin}{{\rm arg}\!\min}

\newcommand{\nexto}{\kern -0.54em}

\newcommand{\dZ}{{\cal Z \kern -0.7em Z}}
\newcommand{\dC}{{\rm\hbox{C \kern-0.8em\raise0.2ex\hbox{\vrule height5.4pt width0.7pt}}}}
\newcommand{\dQ}{{\rm\hbox{Q \kern-0.85em\raise0.25ex\hbox{\vrule height5.4pt width0.7pt}}}}

\newcommand{\HH}{\mathcal{H}}
\newcommand{\PP}{\mathcal{P}}

\newcommand{\ZZ}{\mathbb{Z}}

\providecommand{\U}[1]{\protect\rule{.1in}{.1in}}
\providecommand{\U}[1]{\protect\rule{.1in}{.1in}}
\providecommand{\U}[1]{\protect\rule{.1in}{.1in}}

\usepackage[disable,colorinlistoftodos,prependcaption,textsize=tiny]{todonotes}

\begin{document}
\mainmatter
\title{Comparing Averaged Relaxed Cutters and Projection Methods: Theory and Examples}

\titlerunning{Cutters vs. Projections}

\author{R. D\'iaz Mill\'an\inst{1} \and Scott B. Lindstrom\inst{2} \and Vera Roshchina\inst{3}}


\institute{Federal Institute of Goi\'as and IME Federal University of Goi\'as, Brazil\\
	\email{rdiazmillan@gmail.com}
	\and
	CARMA, University of Newcastle, Australia\\
	\email{scott.lindstrom@uon.edu.au}
	\and 
	RMIT University, Australia\\
	\email{vera.roshchina@rmit.edu.au}
	}

\maketitle

\begin{abstract}
We focus on the convergence analysis of averaged relaxations of cutters, specifically for variants that---depending upon how parameters are chosen---resemble \emph{alternating projections}, the \emph{Douglas--Rachford method}, \emph{relaxed reflect-reflect}, or the \emph{Peaceman--Rachford} method. Such methods are frequently used to solve convex feasibility problems. The standard convergence analyses of projection algorithms are based on the \emph{firm nonexpansivity} property of the relevant operators. However if the projections onto the constraint sets are replaced by cutters (projections onto separating hyperplanes), the firm nonexpansivity is lost. We provide a proof of convergence for a family of related averaged relaxed cutter methods under reasonable assumptions, relying on a simple geometric argument. This allows us to clarify fine details related to the allowable choice of the relaxation parameters, highlighting the distinction between the exact (firmly nonexpansive) and approximate (strongly quasi-nonexpansive) settings. We provide illustrative examples and discuss practical implementations of the method. 
\end{abstract}

\section{Introduction}

Projection and reflection methods are used for solving the \emph{feasibility} problem of finding a point in the intersection of a finite collection of closed, convex sets in a Hilbert space. Such problems have a wide range of application in variational analysis, optimisation, physics and mathematics in general.  One of the most successful methods from this class is the Douglas--Rachford method that uses a combination of reflections and averaging on each iteration. The idea first appeared in \cite{DR} as a numerical scheme for solving differential equations, and the convergence of a more general scheme for finding a zero of the sum of two maximally monotone operators was framed in \cite{LM} (also see \cite[Chapter 26]{Bau} for a modern treatment). Arag\'on Artacho and Campoy have recently introduced a modification of the Douglas--Rachford method for finding closest feasible points \cite{FranNewMethod}.

Convergence rates for such methods are the subject of extensive research; we provide a brief sampling. Under appropriate conditions, the Douglas--Rachford method converges in finitely many steps \cite{Bau}. Convergence rates may frequently be obtained through analysis of regularity conditions \cite{subt}. Additionally, semialgebraic structure admits further bounds on convergence rates for projection methods more generally \cite{DLW,AltSemialg,BLT2015} and for the Douglas--Rachford method in particular \cite{LP}. For a recent survey on the Douglas--Rachford method, see \cite{DRsurvey}.

The idea of replacing projections with their approximations, and specifically with the approximations constructed from the subdifferentials of the convex functions that describe the sets, was introduced by Fukushima \cite{Fukushima}. It has been used in various contexts recently, including the numerical solution of variational inequalities; see, for example, \cite{BelloCruzIusem,RelaxedProjection}. In particular, Combettes has used relaxation parameters together with subgradient projections in the construction of his \emph{extrapolation method of parallel subgradient projections} (EMOPSP) algorithm for image recovery \cite{combettes1997convex}. Of particular relevance are works of Arkady, Cegielski, Reich, and Zalas \cite{CRZ, reic, Cegielski}. Books which contain useful information about general cutters are those of Cegielski \cite{Cegielski}, Censor and Zenios \cite{censor1997parallel}, and Polyak \cite{polyak1987introduction}. 

The subgradient projector in particular is quite well studied; early contributions include the foundational work of Polyak \cite{polyak1987introduction} on subgradient projections and the analysis of the cyclic version by Censor \cite{censor1982cyclic}. Bauschke, C. Wang, X. Wang, and J. Xu provided characterizations of finite convergence in \cite{BWWX1} and a systematic study of the subgradient projector in \cite{BWWX2,BWWX3}. This sampling of the literature on subgradient projections is far from exhaustive; the interested reader is referred to  literature referenced in the latter works of Bauschke et al.

We consider the $2$ set feasibility problem of finding 
\begin{equation}\label{eq:feas}
u \in A \cap B
\end{equation}
for closed, convex subsets $A$ and $B$ of a Hilbert space $\HH$. In particular, we consider the behaviour of the dynamical systems which arise from iterated application of an operator $T$ that is a weighted average of the identity map and the composition of two relaxed cutters for the two sets in question. The Douglas--Rachford method (reflect-reflect-average), the Peaceman--Rachford method, alternating projections, and relaxed-reflect-reflect (RRR) are all special cases.

The \emph{goal} of the present work is threefold:
\begin{enumerate}
	\item We compare and contrast what is true of such operators in the special case where the cutters are projections (onto the constraint sets) with the more general case of cutters.
	\item In particular, we discuss nonexpansivity in the former setting and quasinonexpansivity in the latter, analysing what may be shown through each.
	\item We illustrate with examples, and we provide simple geometric arguments throughout the exposition.
\end{enumerate}

We would also like to highlight the recent work of Jonathan Borwein and his collaborators, who successfully applied the Douglas--Rachford method to a range of large-scale \emph{nonconvex} problems and studied its convergence \cite{DRMxcomp,ABT,JonBrailey,FranMattJonGlobal,BLSSS}. In Borwein's chapter of \emph{Tools and Mathematics: Instruments for Learning} \cite{hhm}, he included, along with his own commentary, a quote he particularly liked:
\begin{displayquote}
	Long before current graphic, visualisation and geometric tools were available,
	John E. Littlewood, 1885-1977, wrote in his delightful Miscellany:
	\begin{displayquote}
		A heavy warning used to be given [by lecturers] that pictures are not rigorous; this has never had its bluff called and has permanently frightened its victims into playing for safety. Some pictures, of course, are not rigorous, but I should say most are (and I use them whenever possible myself).\cite[p.53]{Littlewood}
	\end{displayquote}
\end{displayquote}

In this spirit, we present our results in a tutorial form, complete with many pictures and examples that highlight the geometric intuition underpinning them.

\subsubsection*{Outline}
We introduce the main concepts in Section~\ref{sec:prelims}. In Section~\ref{sec:main} we provide a simple proof of convergence to a feasible point for a method which averages the composition of two relaxed cutters with the identity. The parameterized method recovers alternating projections as one special case and the Douglas--Rachford method as a limiting, but not allowable, case. This comes as no surprise since projections onto constraint sets are a special case of cutters, and examples where the Douglas--Rachford method converges to fixed points which are not also feasible points are well known. 

With projections onto the constraint sets, the fixed points of the Douglas--Rachford operator have the handy property that they may be used to find feasible points in a single step. In Examples~\ref{eg:fixedpoints1} and~\ref{eg:fixedpoints2}, we show this may fail when projections onto the constraint sets are replaced with more general cutters. The elegance of this pairing is that the geometry illustrates why the proof fails if the limiting parameters are allowed, and the examples showcase what can then go wrong.

In Section~\ref{sec:imp} we provide several examples of implementations of the Douglas--Rachford method with cutters. \todo{Explain how exactly our main results relate to \cite{CRZ,Cegielski}. SKIPPING FOR NOW}

\section{Background and preliminaries}\label{sec:prelims}

Let $A$ and $B$ be two closed convex sets in a Hilbert space $\HH$. Given a starting point $x_0\in \HH$, the classic  method of alternating projections generates the sequence of points $\{x_n\}_{n \in \NN}$, where 
\begin{equation}\label{eq:alt-step}
x_n := (P_B \circ P_A)^n x_0 \quad \forall n \in \NN,
\end{equation}
and by $P_S$ we denote the Euclidean projection operator onto a closed convex set $S\subset \HH$, 
$$
P_S(x) = \argmin_{s\in S}\|s-x\|,
$$
which is well-defined (and single valued) for $S$ closed, convex, and nonempty. We assume these properties for all of our sets throughout. Observe (see Figure~\ref{fig:APandDR}~(a)) that each iteration of the method is the composition of projections onto the hyperplanes $H_A$ and $H_B$ that support the sets $A$ and $B$ at $P_A(x)$ and $P_B(x)$ respectively.
\begin{figure}[ht]
	\begin{subfigure}{.45\textwidth}		
	\begin{tikzpicture}[scale=3.0]
	\fill[cyan,fill opacity=0.4] (1-.25,1-.1) to (2.0-.25,0-.1)
	to [out=180,in=270] (.465-.25,.415-.1)
	to [out=90,in=205] cycle;
	\node at (1.1-.25,.2-.1) {$A$};
	
	
	\fill[red,fill opacity=0.4] (1.42-.2,.3+.2) to (1.2-.2,.3+.2) to (1.3-.2,.5+.2) to (1.5-.2,.9+.2) to [out=60,in=135] (1.9-.2,.9+.2) to [out=315,in=0] (1.5-.2,.3+.2) to cycle;
	\node [] at (1.6,.9) {$B$}; 					
	
	
	
	\draw [gray, dashed] (.625,1.25) -- (0,0);
	\node [above right] at (.625,1.25) {$H_A$};
	
	\draw [fill,black] (-.25,.75) circle [radius=0.015];	
	\node [left] at (-.25,.75) {$x$};
	\draw [black,->] (-.2,.725) -- (.2,.525);
	\draw [fill,black] (.25,.5) circle [radius=0.015];
	\node [left] at (.25,.45) {$P_Ax$};

	\draw [black,->] (.325,.5) -- (.925,.5);
	\draw [fill,black] (1,.5) circle [radius=0.015];
	\node [below left] at (1,.5) {$P_B P_A x$};
	
	\draw [gray, dashed] (1,1.25) -- (1,0);
	\node [right] at (1,1.25) {$H_B$};
	
	\end{tikzpicture}
	\caption{One step of alternating projections}\label{fig:APandDRleft}
	\end{subfigure}
	\begin{subfigure}{.45\textwidth}
		\begin{tikzpicture}[scale=3.0]
		\fill[cyan,fill opacity=0.4] (1-.25,1-.1) to (2.0-.25,0-.1)
		to [out=180,in=270] (.465-.25,.415-.1)
		to [out=90,in=205] cycle;
		\node at (1.1-.25,.2-.1) {$A$};
		
		
		\fill[red,fill opacity=0.4] (1.42-.2,.3+.2) to (1.2-.2,.3+.2) to (1.3-.2,.5+.2) to (1.5-.2,.9+.2) to [out=60,in=135] (1.9-.2,.9+.2) to [out=315,in=0] (1.5-.2,.3+.2) to cycle;
		\node [] at (1.65,1) {$B$}; 					
		
		
		
		\draw [gray, dashed] (.625,1.25) -- (0,0);
		\node [above right] at (.625,1.25) {$H_A$};
		
		\draw [fill,black] (-.25,.75) circle [radius=0.015];	
		\node [left] at (-.25,.75) {$x$};
		\draw [black,->] (-.2,.725) -- (.2,.525);
		\draw [fill,black] (.25,.5) circle [radius=0.015];
		\node [left] at (.25,.45) {$P_Ax$};
		\draw [blue,->] (.3,.475) -- (.7,.275);
		\draw [fill,blue] (.75,.25) circle [radius=0.015];	
		\node [left, blue] at (.75,.2) {$R_A^{\gamma=0} x$};

		\draw [blue,->] (1.05,.55) -- (1.2,.7);
		\draw [fill,blue] (1.25,.75) circle [radius=0.015];
		\node [above,blue] at (1.25,.75) {$R_B^{\gamma=0} R_A^{\gamma=0} x$};
		\draw [black,->] (.8,.3) -- (.95,.45);
		\draw [fill,black] (1,.5) circle [radius=0.015];
		\draw [blue,->] (1.05,.55) -- (1.2,.7);
		\node [right] at (1,.5) {$P_B R_A^{\gamma=0} x$};
		
		\draw [gray, dashed] (.25,1.25) -- (1.6,-.1);
		\node [above left] at (.25,1.25) {$H_B$};
		
		\draw [purple, <->] (-.2,.75) -- (1.2,.75);
		\draw [fill,purple] (.5,.75) circle [radius=0.015];
		\node [below,purple] at (.5,.75) {$T_{A,B}x$};
		
		\end{tikzpicture}
		\caption{One step of Douglas--Rachford method}\label{fig:APandDRright}
	\end{subfigure}
	\caption{The operator $T_{A^\gamma,B^\gamma}^\lambda$ for different values of $\gamma, \lambda$.}\label{fig:APandDR}
\end{figure}
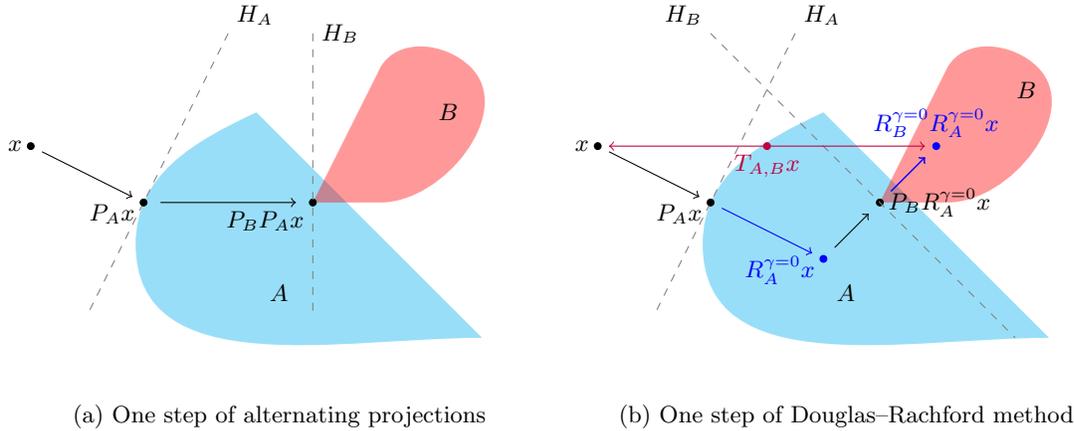

On each step of the classic Douglas--Rachford algorithm the previous iterate is first \emph{reflected} through $H_A$, then reflected through  $H_B$, and finally the resulting point is averaged with the previous iterate; see Figure~\ref{fig:APandDR}~(b). In this case our iterative sequence $\{x_n\}_{n \in \NN}$ is defined as 
\begin{equation}\label{eq:DR-step}
x_n := \left(\frac{1}{2}\left((2P_B-\Id) \circ (2P_A-\Id)\right)+\frac{1}{2}\Id \right)^n x_0  \quad \forall n \in \NN.
\end{equation}

The reflection can be replaced by a \emph{relaxed projection} which we denote by $R_S^\gamma$. For a fixed \emph{reflection parameter} $\gamma \in [0,2)$ we let
\begin{equation}\label{eq:defRS}
R_S^\gamma :=(2-\gamma)(P_S -\Id)+\Id.
\end{equation}
Observe that when $\gamma = 0$, the operator $R_S^{\gamma=0} = 2 P_S -\Id$ is the standard \emph{reflection} that we saw earlier, for $\gamma=1$ we obtain the \emph{projection}, $R_S^{\gamma=1} =  P_S$. For $\gamma\in (1,2)$ the operator $R_S^\gamma$ can be called an \emph{under-relaxed projection} following \cite{DePierro}. For $\gamma\in (0,1)$ it may be called an \emph{over-relaxed projection}.

In addition to using  relaxed projections as in \eqref{eq:defRS}, the averaging step of the Douglas--Rachford iteration \eqref{eq:DR-step} can also be relaxed by choosing an arbitrary point on the interval between the second reflection and the initial iterate. This can be parametrised by some $\lambda \in (0,1]$. We can hence define a $\lambda$-averaged relaxed sequence $\{x_n\}_{n \in \NN}$ by
\begin{align}
x_n &:= \left(T_{A^\gamma,B^\mu}^\lambda \right)^n x_0, \nonumber \\
\text{where}\quad T_{A^\gamma,B^\mu}^\lambda &:= \lambda(R_B^\mu \circ R_A^\gamma)+(1-\lambda)\Id. \label{DRsequence}
\end{align}
 When $\lambda=1$ and $\gamma = \mu = 1$, this is the sequence generated by alternating projections \eqref{eq:alt-step}. For $\gamma = \mu = 0$, this is the Douglas--Rachford method \eqref{eq:DR-step}, and for $\lambda = 1$ the Peaceman--Rachford method. The case where $\gamma = \mu = 0$ and $\lambda$ is flexible is often referred to as \emph{relaxed-reflect-reflect} or RRR \cite{elser2017matrix}. If $\gamma = \frac{2(\eta+1)}{2\eta+1}$, then 
 $$
 R_S^\gamma = \left(\frac{1}{2\eta+1}\Id + \frac{2\eta}{2\eta+1}P_S \right)
 $$
 may be recognized as the form in which the relaxation was presented by Borwein, Li, and Tam for their damped Douglas--Rachford variant \cite{BLT2015}.

We note that the framework introduced here does not cover all possible projection methods. For example, one may want to vary the parameters $\gamma$, $\mu$ and $\lambda$ on every step, or consider other variations of Douglas--Rachford-like operators (e.g. see \cite{FranNewMethod}).

We recall the definition of a cutter (see \cite[Definition 2.1.30]{Cegielski}).

\begin{definition}\label{def:cutter}
	Where $x,y \in \HH$, we say that $y$ separates $S$ from $\HH$ if $\langle x-y,z-y \rangle \leq 0$ for all $z \in S$. We call $T:\HH \rightarrow \HH$ a \emph{cutter} if $y:=Tx$ separates $\Fix T \neq \emptyset$ from $x$ for all $x \in H$. In other words,
\begin{equation}\label{eq:fixed-cutter}
	(\forall x \in \HH) \;(\forall z \in \Fix T) \quad \langle x - Tx, z-Tx \rangle \leq 0.
\end{equation}
\end{definition}
A cutter may be thought of as a map which assigns $x$ to its projection onto a chosen separating hyperplane, as illustrated in Figure~\ref{fig:subdifferentialb}. The Euclidean projection operator $P_S$ for a closed, convex set $S$ is an example of a cutter where the separating hyperplane is a supporting hyperplane to $S$, as illustrated in Figure~\ref{fig:APandDR} for alternating projections at left and the Douglas--Rachford method at right.

We note that \eqref{eq:fixed-cutter} is essential for cutter based projection methods. We have the following elementary example that illustrates this.

\begin{example} In the one-dimensional real setting assume that $S = (-\infty,0]$ and 
$$
T(x) = 
\begin{cases}
x, & x\in (-\infty, 0],\\
0, & x\in (0,1),\\
1, & x\in [1,+\infty).
\end{cases}
$$
Observe that $y = T(x)$ is a separator, however, it is not a cutter: the point $x=1\notin S$ is a fixed point of $T$, and for $x\in (0,1)$ the point $T(x) =0$ does not separate the fixed points of $T$ from $x$. 
\end{example}

A useful implementation of a cutter is the subgradient projection operator for a convex function $f$, which we recall in the following definition from \cite[Definition 2.2]{BWWX3}, where $\partial f$ denotes the usual Moreau-Rockafellar subdifferential of $f$.

\begin{definition}\label{def:subgradientprojector}Let $f: \HH \rightarrow \RR$ be lower semicontinuous and subdifferentiable. Let $s: \HH \rightarrow \RR$ be a selection for $\partial f$. Then the \emph{subgradient projector} of $f$ is
	\begin{equation}\label{eq:subgr}
	P_{\partial f}: \HH \rightarrow \HH : x \mapsto \begin{cases}x-\frac{f(x)}{\|s(x)\|^2}s(x) & \text{ if } f(x) > 0; \\
	x & \text{otherwise}.\end{cases}
	\end{equation}
\end{definition}
The subgradient projection operator is a cutter with $\Fix P_{\partial f} = \lev_{\leq 0}f$. We illustrate in Figure~\ref{fig:subdifferential}. In Figure~\ref{fig:subdifferentialright} we show the case where the selection operator $s$ is uniquely determined since $\partial f$ is single-valued everywhere. In Figure~\ref{fig:subdifferentialleft} we show two possible values for the subgradient projection of $x$; we emphasize that the subgradient projector a is \emph{single-valued operator}, and that the output depends on the chosen selection operator $s$ in Definition~\ref{def:subgradientprojector}.

	\begin{figure}[h]
		\begin{subfigure}{.3\textwidth}
			\begin{adjustbox}{trim=.5cm .4cm .85cm 2.5cm,clip=true}
				\begin{tikzpicture}[scale=2.5]
				\draw [gray] (-.5,0) -- (2,0);
				\draw [black] (-.5,-.25) -- (1,.5);
				\draw [black] (1,.5) -- (2,2.5);
				\draw [red,dashed] (.625,-.35) -- (1.5,1.4); 
				\draw [blue,dashed] (0,-.05) -- (2,.95); 
				\draw [fill,black] (1,0) circle [radius=.025];
				\node [below] at (1,-.05) {$x$};
				\draw [purple,->] (1,.1) -- (1,.4);
				\draw [purple,->] (.95,.375) -- (.45,.065);
				\draw [fill,purple] (.4,0) circle [radius=.025];
				\draw [purple,->] (.95,.365) -- (.65,.07);
				\draw [fill,purple] (.6,0) circle [radius=.025];
				\end{tikzpicture}
			\end{adjustbox}
			\caption{Subgradient projection}\label{fig:subdifferentialleft}
		\end{subfigure}	
	\begin{subfigure}{.3\textwidth}
	\begin{tikzpicture}[scale=3.0]
	
	\fill[cyan,fill opacity=0.4] (1,1) to (1.5,0)
	to [out=180,in=270] (.5,.5)
	to [out=90,in=205] cycle;
	\node at (.9,.8) {$S$};
	
	
	\draw [gray, dashed] (.625,1.25) -- (0,0);
	\node [above right] at (.625,1.25) {$H$};
	
	\draw [fill,black] (-.25,.75) circle [radius=0.015];	
	\node [left] at (-.25,.75) {$x$};
	\draw [black,->] (-.2,.725) -- (.2,.525);
	\draw [fill,black] (.25,.5) circle [radius=0.015];
	\node [left] at (.25,.45) {$\PP_Sx$};
	\draw [blue,->] (.3,.475) -- (.7,.275);
	\draw [fill,blue] (.75,.25) circle [radius=0.015];	
	\node [left, blue] at (.75,.2) {$\mathcal{R}^{\gamma=0}_Sx$};
	
	\end{tikzpicture}
	\caption{Cutter reflection}\label{fig:subdifferentialb}
	\end{subfigure}\hspace*{\fill}
	\begin{subfigure}{.3\textwidth}
		\begin{adjustbox}{trim=0cm 0.5cm 0cm 1.0cm,clip=true}
			\begin{tikzpicture}
			[scale=0.95]
			\node[anchor=south west,inner sep=0] (image) at (0,0) {\includegraphics[width=1\linewidth]{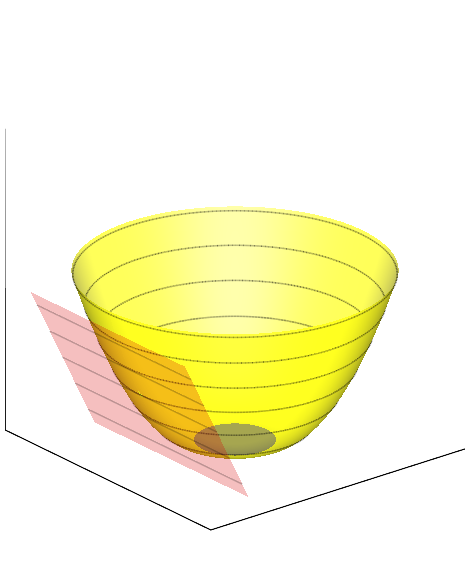}};
			\begin{scope}[x={(image.south east)},y={(image.north west)}]
			
			\draw [fill,red] (.28,.195) circle [radius=.075cm];
			\node [above left,red] at (.28,.195) {$x$};
			
			\draw [fill,red] (.3475,.2125) circle [radius=.075cm];
			\draw [red,thick,->] (.28,.215) -- (.28,.3575);
			\draw [red,<-,thick] (.337375,.23425) -- (.290125,.33575);

			\end{scope}
			\end{tikzpicture}
		\end{adjustbox}
		\caption{Subgradient projection}\label{fig:subdifferentialright}
	\end{subfigure}

	\caption{Subgradient projections are cutters.}\label{fig:subdifferential}
\end{figure}
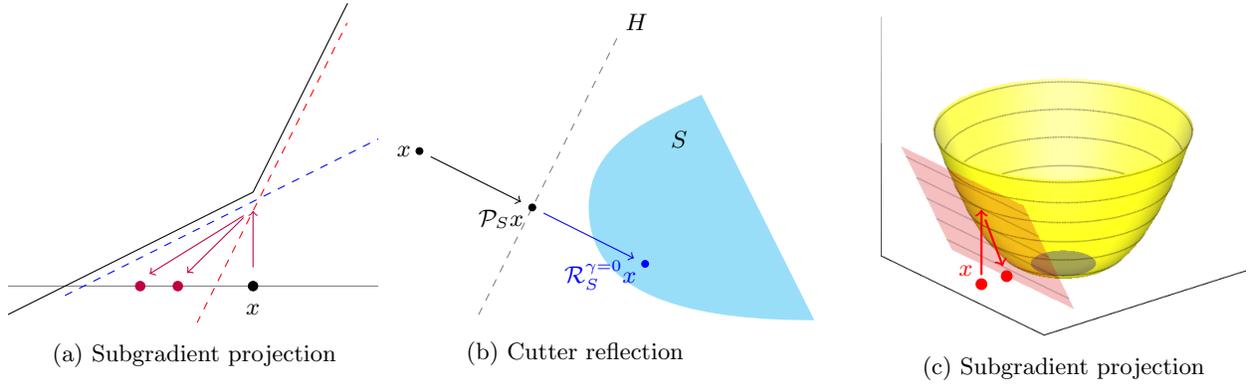

In the case where projections onto the sets cannot be computed (or computing them exactly is undesirable), it makes sense to consider operators of the form \eqref{DRsequence} where the projections are replaced with subgradient projections or other kinds of cutters.

We will refer to all such discussed methods and their combination as \emph{cutter methods} and use the notation 
$$
\mathcal{T}_{A^\gamma,B^\mu}^\lambda := \lambda(\mathcal{R}_B^\mu \circ \mathcal{R}_A^\gamma)+(1-\lambda)\Id,
$$
where 
$$
\mathcal{R}_A^\gamma := (2-\gamma)(\mathcal{P}_A-\Id)+\Id, \quad \mathcal{R}_B^\mu := (2-\mu)(\mathcal{P}_B-\Id)+\Id
$$ 
are the relaxed versions of the cutters $\mathcal{P}_A$ and $\mathcal{P}_B$, which may be projections onto the constraint sets or more general cutters, depending on the context. 

In the case of subgradient projections we will slightly abuse the notation and let 
$$
\mathcal{T}_{f^\gamma,g^\mu}^\lambda : = \mathcal{T}_{(\lev_{\leq 0}f)^\gamma,(\lev_{\leq 0}g)^\mu}^\lambda,
$$
with cutters implemented via the subgradient projections~\eqref{eq:subgr}.

Notice that if for some closed convex set $S$ we let $f:=\dm(\cdot,S)$ be the distance function for the set $S$ given by
$$
\dm(x,S) = \min_{y\in S}\|x-y\|,
$$
then $P_S$ and $P_f$ coincide. We will mainly focus on averaged cutter relaxations $\mathcal{T}_{A^\gamma,B^\mu}^\lambda$, for which an example is shown in Figure~\ref{fig:averaged_cutter_relaxation}, and will elaborate on the functional implementation in Section~\ref{sec:imp}.

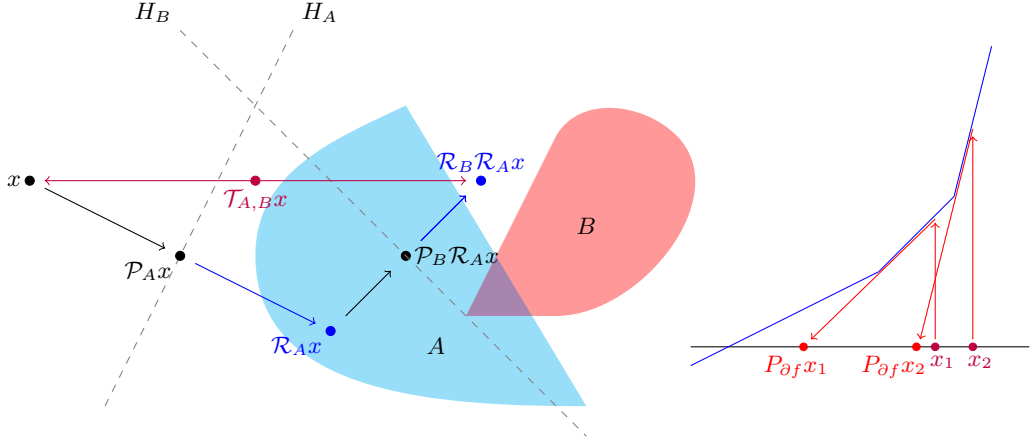
\begin{figure}[ht]
	\begin{center}
	\begin{subfigure}{.55\textwidth}
	\begin{tikzpicture}[scale=4.0]
	\fill[cyan,fill opacity=0.4] (1,1) to (1.6,0)
	to [out=180,in=270] (.5,.5)
	to [out=90,in=205] cycle;
	\node at (1.1,.2) {$A$};
	
	\fill[purple,fill opacity=0.4] (1.42,.3) to (1.2,.3) to (1.3,.5) to cycle;
	
	\fill[red,fill opacity=0.4] (1.42,.3) to (1.3,.5) to (1.5,.9) to [out=60,in=135] (1.9,.9) to [out=315,in=0] (1.5,.3) to cycle;
	\node [] at (1.6,.6) {$B$}; 					
	
	
	\draw [gray, dashed] (.625,1.25) -- (0,0);
	\node [above right] at (.625,1.25) {$H_A$};
	
	\draw [fill,black] (-.25,.75) circle [radius=0.015];	
	\node [left] at (-.25,.75) {$x$};
	\draw [black,->] (-.2,.725) -- (.2,.525);
	\draw [fill,black] (.25,.5) circle [radius=0.015];
	\node [left] at (.25,.45) {$\PP_Ax$};
	\draw [blue,->] (.3,.475) -- (.7,.275);
	\draw [fill,blue] (.75,.25) circle [radius=0.015];	
	\node [left, blue] at (.75,.2) {$\mathcal{R}_Ax$};

	\draw [blue,->] (1.05,.55) -- (1.2,.7);
	\draw [fill,blue] (1.25,.75) circle [radius=0.015];
	\node [above,blue] at (1.25,.75) {$\mathcal{R}_B \mathcal{R}_A x$};
	\draw [black,->] (.8,.3) -- (.95,.45);
	\draw [fill,black] (1,.5) circle [radius=0.015];
	\draw [blue,->] (1.05,.55) -- (1.2,.7);
	\node [right] at (1,.5) {$\mathcal{P}_B \mathcal{R}_A x$};
	
	\draw [gray, dashed] (.25,1.25) -- (1.6,-.1);
	\node [above left] at (.25,1.25) {$H_B$};
	
	\draw [purple, <->] (-.2,.75) -- (1.2,.75);
	\draw [fill,purple] (.5,.75) circle [radius=0.015];
	\node [below,purple] at (.5,.75) {$\mathcal{T}_{A,B}x$};
	
	\end{tikzpicture}
	\end{subfigure}
	\begin{subfigure}{.4\textwidth}
		\begin{tikzpicture}[scale=1]
		\draw [black] (-.5,0) -- (4,0);
		\draw [blue] (-.5,-.25) -- (2,1);
		\draw [blue] (2,1) -- (3,2);
		\draw [blue] (3,2) -- (3.5,4);
		
		\draw [fill,purple] (3.25,0) circle [radius=0.05];
		\node [below,purple] at (3.35,0) {$x_2$};
		\draw [red,->] (3.25,.1) -- (3.25,2.8);
		\draw [red,->] (3.25,2.9) -- (2.55,.1);					
		\draw [fill,red] (2.5,0) circle [radius=0.05];
		\node [below, red] at (2.2,0) {$P_{\partial f} x_2$};			
		
		\draw [fill,purple] (2.75,0) circle [radius=0.05];
		\node [below,purple] at (2.85,0) {$x_1$};				
		\draw [red,->] (2.75,.1) -- (2.75,1.65);
		\draw [red,->] (2.75,1.7) -- (1.1,.1);			
		\draw [fill,red] (1,0) circle [radius=0.05];
		\node [below,red] at (.9,0) {$P_{\partial f} x_1$};							
		\end{tikzpicture}
	\end{subfigure}

	\end{center}
	\caption{An averaged relaxed cutter $\mathcal{T}_{A^\gamma,B^\mu}^\lambda$ may not be nonexpansive.}\label{fig:averaged_cutter_relaxation}
\end{figure}

Let $\A = N_A$ and $\B = N_B$ be the normal cone operators for closed convex sets $A$ and $B$. Then the resolvents $J_\A^\lambda,J_\B^\lambda$ (defined as $J_F^\lambda = (\Id+\lambda F)^{-1}$ for some set-valued mapping $F$) are the projection operators $P_A,P_B$ respectively, $T_{\A,\B} = \frac{1}{2}R_B^{\gamma=0} R_A^{\gamma=0} + \frac{1}{2}\Id$ is what we recognize as the Douglas--Rachford method, and $J_\A^\lambda v = \proj_A v \in A \cap B$ is a solution for the feasibility problem.

We quote the following key result from \cite{LM} that applies to a more general setting of maximal monotone operators. 
\begin{theorem}[Lions \& Mercier]\label{thm:LionsandMercier} Assume that 
	 $\A,\B$ are maximal monotone operators and $\A + \B$ is maximal monotone. Then for
	\begin{equation}
	T_{\A,\B}:\HH \rightarrow \HH \text{ by } x\mapsto J_\B^\lambda(2J_\A^\lambda-\Id)x+(\Id-J_\A^\lambda)x
	\end{equation}
	the sequence given by $x_{n+1}=T_{\A,\B}x_n$ converges weakly to some $v \in \HH$ as $n\rightarrow \infty$ such that $J_\A^\lambda v$ is a zero of $\A + \B$.
\end{theorem}
Bauschke, Combettes, and Luke \cite{BCL} showed that in the case of the feasibility problem \eqref{eq:feas} the requirement $\A + \B$ maximal monotone may be relaxed, a relaxation later made more general by Svaiter \cite{Svaiter}. See also \cite[Theorem 26.11]{BCL}. Both results rely on the firm nonexpansivity of $T_{\A,\B}$, an immediate consequence of the fact that $R_\B^{\gamma=0} R_\A^{\gamma=0}$ is \emph{nonexpansive} and so $T_{\A,\B}$ is $1/2$-averaged. We define this term and several others which we summarise in the following definition (see \cite[Def 4.1]{BC}, \cite[Def 2.2]{CRZ}, and \cite[Def 2.1.19]{Cegielski} for more details).

\begin{definition}[Properties of operators]\label{def:SQNE}
	Let $D\subset \HH$ be nonempty and let $T:D\rightarrow \HH$. Assume that $\Fix T := \{x \in \HH\, |\, Tx = x\} \ne \emptyset$. Then $T$ is

\noindent\textbf{firmly nonexpansive} if
		\begin{equation*}
		\|T(x)-T(y)\|^2 + \|(\Id-T)(x)-(\Id-T)(y)\|^2 \leq \|x-y\|^2 \qquad \forall x \in D, \quad \forall y \in D;
		\end{equation*}

\noindent\textbf{nonexpansive} if it is Lipschitz continuous with constant 1,
		\begin{equation*}
		\|T(x)-T(y)\| \leq \|x-y\| \qquad \forall x \in D, \quad \forall y \in D;
		\end{equation*}

\noindent\textbf{quasinonexpansive} if $ \qquad \|T(x)-y\| \leq \|x-y\| \qquad \forall x \in D, \quad \forall y \in \Fix T $

\noindent (an operator that is both quasinonexpansive and continuous is called paracontracting);

\noindent\textbf{strictly quasinonexpansive}  if
		\begin{equation*}
		\|T(x)-y\| < \|x-y\| \qquad \forall x \in D \setminus \Fix T,\quad \forall y \in \Fix T;
		\end{equation*}
\noindent\textbf{$\rho$-strongly quasinonexpansive}  for $\rho > 0$ if $$
		 \|Tx-y\|^2 \leq \|x-y\|^2 - \rho \|Tx -x\|^2 \qquad \forall x \in D \setminus \Fix T,\quad \forall y \in \Fix T.
		$$

\end{definition}

We are focussed on the feasible setting, so we can safely assume that for all operators $T$ considered in the paper $\Fix T \ne \emptyset$. As soon as one moves from the setting of projections into the setting of more general cutters, the (firmly) nonexpansive property of $\mathcal{T}_{A^\gamma,B^\mu}^\lambda$ may be lost, as illustrated in the following simple example.

\begin{example}[Loss of nonexpansivity when using cutters]\label{eg:1}Define $f:\RR \rightarrow \RR$ by
	\begin{equation}\label{eq:eg1f}
	f: x\mapsto \begin{cases}
	|x| & x\leq 1,\\
	2x-1 & \text{otherwise}.
	\end{cases}
	\end{equation}
	Then the subgradient cutter $P_{\partial f}:\RR\to \RR$ for the level set $\lev_{\leq 0} f$ is 
	\begin{equation}\label{eq:eg1P}
	P_{\partial f}: x\mapsto \begin{cases}
	0 & x <1,\\
	\frac{1}{2} & x>1,\\
	\text{some } u \in [0,1/2]& x=1.
	\end{cases}
	\end{equation}
	Observe that $P_{\partial f}$ is not nonexpansive for any choice of $x \in (0,1), y \in (1,2)$ satisfying $|x-y|<\frac{1}{2}$. A similar polyhedral example is shown at right in Figure~\ref{fig:averaged_cutter_relaxation}.
\end{example}

Strong quasinonexpansivity is a less restrictive property that yields the desired convergence, though under a slightly more restrictive parameter scheme.
\begin{definition}[Fej\'er monotonicity]A sequence $(x_n)_{n\in \NN}$ is Fej\'er monotone with respect to closed convex set $C$ if
	$$
	\|x_{n+1}-x\|\leq \|x_n - x\|\qquad \forall x\in C, \quad \forall n \in \NN.
	$$
\end{definition}
A Fej\'er monotone sequence with respect to a closed convex set $C$ may be thought of as a sequence defined by $x_n := T^n x_0$ where $T$ is QNE with respect to $C=\Fix T$. Note that a Fej\'er monotone sequence with respect to a non-empty set is always bounded.

We have the following well-known convergence result (see \cite[Theorem 5.11]{BC}).

 \begin{theorem}\label{baucom}
 Let $(x_n)_{n\in\mathbb{N}}$ be a sequence in $\HH$ and let $C$ be a nonempty closed convex subset of $\HH$. Suppose that $(x_n)_{n\in\mathbb{N}}$ is Fej\'er monotone with respect to $C$. Then the following are equivalent: 
 \begin{enumerate}
 \item the sequence $(x_n)_{n\in\mathbb{N}}$ converges strongly (i.e. in norm) to a point in $C$;
 \item $(x_n)_{n\in\mathbb{N}}$ possesses a strong sequential cluster point in $C$;
 \item $\liminf\limits_{n\to \infty} \dm(x_n,C)=0.$
 \end{enumerate}
 \end{theorem}
	
\section{Convergence of Projection Methods}\label{sec:main}

In the following theorem, \ref{lem_2} is a known consequence of \cite[Corollary 3.7.1(i)]{Cegielski}. However, we provide a new proof which relies on simple geometry. We will then go on to analyse convergence for $\mathcal{T}_{A^\gamma,B^\mu}^\lambda$, and the details of our proof will illustrate why for averaged cutter relaxation methods we may lose convergence in the case of $\gamma=0$.

\begin{theorem}\label{thm:sqne} Let $A$ be a closed convex set in a Hilbert space $\HH$, and let $\mathcal{P}_A$ be a cutter. Then the following hold:
	\begin{enumerate}[label={(\roman*)}]
		\item $ \|\mathcal{R}_{A}^\gamma(x)-y\|^2 \leq \gamma (\gamma-2) \|x-\mathcal{P}_A(x)\|^2+\|x-y\|^2 \qquad \forall y \in A\quad  \forall x \in \HH$; \label{lem_2}
		\item $\mathcal{R}_A^\gamma$ is $\gamma/(2-\gamma)$-strongly quasinonexpansive;\label{lem_SQNE}
		\item $\mathcal{R}_{A}^\gamma$ is strictly quasinonexpansive for $\gamma \in (0,2)$ if $(\forall x \notin A)\; \mathcal{P}_A(x)\neq x$. \label{lem_4}
	\end{enumerate}
\end{theorem}
\begin{proof}
	 If $x \in A$, then $\mathcal{R}_A^\gamma(x)=\PP_A(x)=x$ and the proof of \ref{lem_2} is trivial. Consider the case when  $x \notin A$. 	Without loss of generality we can assume that $x=0$. Indeed, it is evident that for the affine change of variable $u' = u-x$ the induced mapping $\mathcal{P}_{A'}(u') = \PP_{A-x}(u-x)$ is again a cutter for $A' = A-x$, and the relation (i) can be restated in terms of $A'$ and $\PP_{A'}$; this is also clear from the geometry illustrated in Figure~\ref{fig:triangle}.

Fix $y \in A$. We have $y := v+u$ where $v\in {\rm span}\{\PP_A(x)\}, u \in {\rm span}\{\PP_A(x)\}^\perp$. We will first show that
\begin{equation}\label{lemma_key_fact}
\|\mathcal{R}_A^\gamma(x)-v\| \leq \|v\| - \min\{\gamma,2-\gamma\} \|\PP_A(x)\|.
\end{equation}
Here Figure~\ref{fig:triangle} is most instructive, both for understanding this inequality and motivating its proof.

Since $\PP_A$ is a cutter, we have 
$$
\langle y,\PP_A(x)\rangle \geq \|\PP_A(x)\|^2 \quad \forall \, y \in A.
$$	 
Furthermore, we have $v = \beta \PP_A(x)$, hence
$$
\beta \|\PP_A(x)\|^2= \langle v, \PP_A(x)\rangle = \langle y-u,\PP_A(x)\rangle  = \langle y,\PP_A(x)\rangle \geq \|\PP_A(x)\|^2,
$$
which yields $\beta\geq 1$ (observe that $\|P_A(x)\|^2>0$ since $x=0\notin A$). Now
\begin{equation}\label{lemma_before_cases}
\|\mathcal{R}_A^\gamma(x)-v\| = \|(2-\gamma)\PP_A(x)-\beta \PP_A(x)\| = |2-\gamma-\beta| \| \PP_A(x)\|.
\end{equation}

Observe that 
\begin{align*}
|2-\gamma - \beta| 
& = \max\{2-\gamma-\beta, \beta+\gamma -2\}  \\
&  = \beta + \max\{2\underbrace{(1-\beta)}_{\leq 0}-\gamma, \gamma-2\} \\
& \leq \beta+\max\{-\gamma, \gamma-2\} \\
& = \beta-\min\{\gamma, 2-\gamma\},
\end{align*}
hence we have  \eqref{lemma_key_fact}.
For convenience, let
\begin{equation}
\psi: [0,2) \rightarrow [0,1] \text{ defined by } \gamma \mapsto \min\{\gamma, 2-\gamma\} = \begin{cases}
\gamma & \text{ if } \gamma \in [0,1],\\
2-\gamma & \text{ if } \gamma \in (1,2).
\end{cases}
\end{equation}
\begin{figure}[ht]
		\begin{subfigure}{0.4\textwidth}
			\begin{tikzpicture}[scale=2.8]
			
			\draw [thick,gray,dashed] (1,0) -- (1,-.85);
			\node [left] at (1,-.5) {$H$};
			
			\draw [thick,black,->] (.05,0) -- (.95,0);
			\draw [fill, black] (0,0) circle [radius=0.015];
			\node [left, black] at (0,0) {$x$};
			
			\draw [thick,blue,->] (1.05,0) -- (1.45,0);
			\draw [fill, black] (1,0) circle [radius=0.015];
			\node [above, black] at (1,0) {$\PP_A(x)$};
			
			\draw [fill, black] (1.5,0) circle [radius=0.015];
			\node [above, black] at (1.5,0) {$\mathcal{R}_A^\gamma(x)$};
			
			\draw [thick,red,<->] (1.55,0) -- (1.95,0);
			
			\draw [fill,red] (2,0) circle [radius=0.015];
			\node [above,red] at (2,0) {$v$};
			
			\draw [thick,red,->] (2,-.05) -- (2,-.7);
			\node [left,red] at (2,-.3) {$u$};
			
			\draw [fill,red] (2,-.75) circle [radius=0.015];
			\node [below,red] at (2,-.75) {$y$};
			
			\draw [fill,red,<->] (.1,-0.0375) -- (1.9,-0.7125);
			
			\draw [fill,red,<->] (1.5375,-0.05625) -- (1.9625,-0.69375);
			\end{tikzpicture}
			\caption{Case 1}		
		\end{subfigure}\begin{subfigure}{0.325\textwidth}
		\begin{tikzpicture}[scale=2.8]
		
		\draw [thick,gray,dashed] (1,0) -- (1,-.85);
		\node [left] at (1,-.25) {$H$};
		
		\draw [thick,black,->] (.05,0) -- (.95,0);
		\draw [fill, black] (0,0) circle [radius=0.015];
		\node [left, black] at (0,0) {$x$};
		
		\draw [thick,blue,->] (1.05,-0.05) -- (1.45,-0.05);
		\draw [fill, black] (1,0) circle [radius=0.015];
		\node [above, black] at (1,0) {$\PP_A(x)$};
		
		\draw [fill, black] (1.5,0) circle [radius=0.015];
		\node [above, black] at (1.5,0) {$\mathcal{R}_A^\gamma(x)$};
		
		\draw [thick,red,<->] (1.45,0) -- (1.3,0);
		
		\draw [fill,red] (1.25,0) circle [radius=0.015];
		\node [above,red] at (1.25,0) {$v$};
		
		\draw [thick,red,->] (1.25,-.05) -- (1.25,-.7);
		\node [left,red] at (1.25,-.3) {$u$};
		
		\draw [fill,red] (1.25,-.75) circle [radius=0.015];
		\node [below,red] at (1.25,-.75) {$y$};
		
		\draw [fill,red,<->] (.075,-0.045) -- (1.175,-0.705);
		
		\draw [fill,red,<->] (1.485,-0.05625) -- (1.315,-0.69375);
		\end{tikzpicture}
		\caption{Case 2}		
	\end{subfigure}\begin{subfigure}{0.275\textwidth}
	\begin{tikzpicture}[scale=2.8]
	
	\draw [thick,gray,dashed] (1,0) -- (1,-.75);
	\node [left] at (1,-.2) {$H$};
	
	\draw [thick,black,->] (.05,0) -- (.95,0);
	\draw [fill, black] (0,0) circle [radius=0.015];
	\node [left, black] at (0,0) {$x$};
	
	\draw [thick,blue,->] (.95,-.05) -- (.5,-.05);
	\draw [fill, black] (1,0) circle [radius=0.015];
	\node [above, black] at (1,0) {$\PP_A(x)$};
	
	\draw [fill, black] (.5,0) circle [radius=0.015];
	\node [above, black] at (.5,0) {$\mathcal{R}_A^\gamma(x)$};
	
	\draw [thick,red,<->] (1.05,0) -- (1.20,0);
	
	\draw [fill,red] (1.25,0) circle [radius=0.015];
	\node [above,red] at (1.25,0) {$v$};
	
	\draw [thick,red,->] (1.25,-.05) -- (1.25,-.7);
	\node [left,red] at (1.25,-.3) {$u$};
	
	\draw [fill,red] (1.25,-.75) circle [radius=0.015];
	\node [below,red] at (1.25,-.75) {$y$};
	
	\draw [fill,red,<->] (.0625,-0.03875) -- (1.1875,-0.73625);
	
	\draw [fill,red,<->] (.575,-0.075) -- (1.175,-0.675);
	\end{tikzpicture}
	\caption{Case 3}		
\end{subfigure}

\caption{Illustrations of the inequality~\eqref{lemma_key_fact} in the proof of  Theorem~\ref{thm:sqne}}\label{fig:triangle}	
\end{figure}
	Having shown that \eqref{lemma_key_fact} is true, the Pythagorean theorem yields
	\begin{equation}\label{lemma_pythag_1}
	\|\mathcal{R}_A^\gamma(x)-v\|^2 = \|y-\mathcal{R}_A^\gamma(x)\|^2 - \| u\|^2.
	\end{equation}
	Together \eqref{lemma_pythag_1} and \eqref{lemma_key_fact} yield
	\begin{align}
	\|y-\mathcal{R}_A^\gamma(x)\|^2 &\leq \left(\|v\|-\psi(\gamma)\|\PP_A(x)\|\right)^2 + \|u\|^2. \label{lemma_paper_3}
	\end{align}
	Now the Pythagorean theorem also yields
	\begin{equation}\label{lemma_pythag_2}
	\|u\|^2 = \|y\|^2-\|v\|^2.
	\end{equation}
	Equations \eqref{lemma_paper_3} and \eqref{lemma_pythag_2} together yield
	\begin{align}
	\|y-\mathcal{R}_A^\gamma(x)\|^2 &\leq \left(\|v\|-\psi(\gamma)\|\PP_A(x)\|\right)^2 + \|y\|^2-\|v\|^2 \nonumber \\
	&= -2\psi(\gamma)\|\PP_A(x)\|\cdot\|v\|+\psi(\gamma)^2\|\PP_A(x)\|^2+\|y\|^2. \label{lemma_paper_6}
	\end{align}
	Now since $\|\PP_A(x)\| \leq \|v\|$,
	\begin{equation}\label{lemma_paper_7}
	-2\psi(\gamma) \|\PP_A(x)\|^2 \geq -2\psi(\gamma) \|\PP_A(x)\|\cdot \|v\|.
	\end{equation}
	Now \eqref{lemma_paper_6} and \eqref{lemma_paper_7} together yield
	\begin{align}
	\|y-\mathcal{R}_A^\gamma(x)\|^2 &\leq -2\psi(\gamma) \|\PP_A(x)\|^2+\psi(\gamma)^2\|\PP_A(x)\|^2+\|y\|^2 \nonumber \\
	&= \psi(\gamma) (\psi(\gamma)-2)\|\PP_A(x)\|^2+\|y\|^2\nonumber \\
	&= \gamma(\gamma-2)\|\PP_A(x)\|^2+\|y\|^2, \label{lemma_paper_8}
	\end{align}
	where the final equality comes from the fact that $\psi(\gamma) (\psi(\gamma)-2)=\gamma(\gamma-2)$. This shows \ref{lem_2}. Now since $\gamma \in (0,2]$ we have that $\gamma(\gamma-2)\leq 0$. Combining with the fact that $\|\mathcal{R}_A^\gamma(x)\| = (2-\gamma)\|\PP_A(x)\|$, we have from \eqref{lemma_paper_8} that
	\begin{align*}
	\|y-\mathcal{R}_A^\gamma(x)\|^2 
	&\leq \gamma(\gamma-2)\left(\frac{\|\mathcal{R}_A^\gamma(x)\|}{2-\gamma} \right)^2+\|y\|^2 \\
	&=-\frac{\gamma}{2-\gamma}\|\mathcal{R}_A^\gamma(x)\|^2+\|y\|^2,
	\end{align*}
	which shows \ref{lem_SQNE}. 
	
	If we have $\gamma \in (0,2)$, $x\notin A$, and $(\forall \; x \notin A)\; \PP_A(x) \neq x$, then $\gamma(\gamma-2)\|\PP_A(x)\|^2 <0$ strictly and so
	\begin{equation*}
	\|y-\mathcal{R}_A^\gamma(x)\|^2 \leq \gamma(\gamma-2)\|\PP_A(x)\|^2+\|y\|^2 < \|y\|^2.
	\end{equation*}
	This shows \ref{lem_4}. 
\end{proof}

\begin{theorem}\label{thm:main}\label{maintheorem2}
The following hold: 
\begin{enumerate}[label={(\roman*)}]
	\item  $\mathcal{T}_{A^\gamma,B^\mu}^\lambda$ is quasinonexpansive; \label{maintheorem1}
	\item if $\mu,\gamma \in (0,2)$ then $\mathcal{T}_{A^\gamma,B^\mu}^\lambda$ is strictly quasinonexpansive and 
$$ 
\underset{n\rightarrow\infty}{\lim}\|x_n-\PP_A(x_n)\| = \underset{n\rightarrow\infty}{\lim}\|x_n-\PP_B \mathcal{R}_A^\gamma(x_n)\|=0.
$$
	
\end{enumerate}
\end{theorem}
\begin{proof}
Fix $y \in A\cap B$. For any $x\in \HH$, we have from Theorem~\ref{thm:sqne}:
\begin{align*}
\|\mathcal{R}_A^\gamma(x)-y\|^2 &\leq \gamma(\gamma-2)\|x-\PP_A(x)\|^2 + \|x-y\|^2\\
\text{and}\quad \|\mathcal{R}_B^\mu \mathcal{R}_A^\gamma(x)-y\|^2 &\leq \mu(\mu-2) \|\PP_B \mathcal{R}_A^\gamma (x)-\mathcal{R}_A^\gamma(x)\|^2 + \|\mathcal{R}_A^\gamma(x)-y\|^2.
\end{align*}
Combining these two inequalities yields
\begin{align}
\|\mathcal{R}_B^\mu \mathcal{R}_A^\gamma(x)-y\|^2 &\leq \theta(x) + \|x-y\|^2 \nonumber\\
\text{where}\quad \theta(x) &=\mu(\mu-2) \|\PP_B \mathcal{R}_A^\gamma (x)-\mathcal{R}_A^\gamma(x)\|^2 + \gamma(\gamma-2)\|x-\PP_A(x)\|^2. \label{eqn:RBRA}
\end{align}
By convexity of $\|\cdot - y\|^2$,
\begin{align}
\|\mathcal{T}_{A^\gamma,B^\mu}^\lambda (x)-y\|^2 &= \|\left(\lambda \mathcal{R}_B^\mu \mathcal{R}_A^\gamma(x)+(1-\lambda)x\right)-y\|^2 \nonumber\\
&\leq \lambda \|\mathcal{R}_B^\mu \mathcal{R}_A^\gamma(x)-y\|^2+(1-\lambda)\|x-y\|^2.\label{eqn:convex}
\end{align}
Combining \eqref{eqn:convex} with \eqref{eqn:RBRA} yields
\begin{align}
\|\mathcal{T}_{A^\gamma,B^\mu}(x)-y\|^2 &\leq \lambda \left(\theta(x) + \|x-y\|^2 \right) + (1-\lambda)\|x-y\|^2 = \lambda\theta(x)+\|x-y\|^2. \label{T_theta}
\end{align}
Now notice that \eqref{T_theta} implies the quasinonexpansiveness of $\mathcal{T}_{A^\gamma,B^\mu}^\lambda$, since $\theta(x) \leq 0$ if $x \notin \Fix \mathcal{T}_{A^\gamma,B^]\mu}^\lambda \supset A \cap B$. If we additionally have $\lambda \in (0,1]$ and $\mu,\gamma \in (0,2)$, then $(x \notin \Fix \mathcal{T}_{A^\gamma,B^\mu}^\lambda) \implies \theta(x)<0$, which shows the strict quasinonexpansivity. 

Now we have that
\begin{equation*}
0 \leq \|x_{n+1}-y\|^2 \leq \|x_0-y\|^2 + \lambda \sum_{j=0}^n \theta(x_j).
\end{equation*}
Since $\gamma(2-\gamma)\leq0$ and $\mu(2-\mu)\leq 0$, we have $\theta(x_j) \leq 0 \; \forall j$. Since $\sum_{j=0}^\infty \theta(x_j)$ is a sum of nonpositive terms and is bounded from below, $\theta(x_j) \rightarrow 0$. In particular, let $\gamma,\mu \in (0,2)$ and we have $\gamma(2-\gamma)<0$ and $\mu(2-\mu)< 0$; combining this with the fact that $\theta(x_j) \rightarrow 0$, we obtain
\begin{align}
\underset{n\rightarrow \infty}{\lim}&\|x_n-\PP_A(x_n)\| = 0, \label{theorem_paper_18a}\\
\text{and }\underset{n\rightarrow \infty}{\lim}&\|\mathcal{R}_A^\gamma(x_n)-\PP_B \mathcal{R}_A^\gamma(x_n)\| = 0. \label{theorem_paper_18b}
\end{align}
Now since $\|x_n - \mathcal{R}_A^\gamma(x_n)\| = (2-\gamma)\| x_n - \PP_A(x_n)\|$, \eqref{theorem_paper_18a} implies that
\begin{equation}
\underset{n\rightarrow \infty}{\lim}\|x_n-\mathcal{R}_A^\gamma(x_n)\| = 0. \label{theorem_paper_18c}
\end{equation}
Now the triangle inequality yields
\begin{equation}
\|x_n-\PP_B \mathcal{R}_A^\gamma(x_n)\| \leq \|x_n-\mathcal{R}_A^\gamma(x_n)\| + \|\mathcal{R}_A^\gamma(x_n)-\PP_B \mathcal{R}_A^\gamma(x_n)\|, \label{theorem_paper_18d}
\end{equation}
and so \eqref{theorem_paper_18c} and \eqref{theorem_paper_18d} together imply
\begin{equation*}
\underset{n\rightarrow \infty}{\lim}\|x_n-\PP_B \mathcal{R}_A^\gamma(x_n)\| = 0.
\end{equation*}
This completes the proof.
\end{proof}

From Theorem~\ref{thm:main} we obtain a number of convergence results.

\begin{theorem}\label{thm:ABtozero}
	Let $\gamma, \mu \in (0,2)$ and $\lambda \in (0,1]$. Suppose that the following hold:
	\begin{enumerate}[label=(\Roman*),ref=(\Roman*)]
		\item \label{Atozero} $\displaystyle\lim_{n\to \infty}\|x_n-\mathcal{P}_A(x_n)\| = 0$ implies $\displaystyle\lim_{n\to \infty} \dm(x_n,A) =0$;
		\item \label{Btozero} $\displaystyle\lim_{n\to \infty}\|x_n-\mathcal{P}_B \mathcal{R}^\gamma_A(x_n)\|=0$ implies $\displaystyle\lim_{n\to \infty}\dm(x_n, B) =0$.  
	\end{enumerate}
	Then $(x_n)_{n\in \mathbb{N}}$ converges weakly to a point in $A\cap B$. Moreover, any one of the three conditions below guarantee that $(x_n)_{n \in \mathbb{N}}$ converges strongly to a point in $A\cap B$:
	\begin{enumerate}[label={(\roman*)}]
		\item $\mathcal{H}$ is finite dimensional \label{3p3_condi}
		\item One of $A$ or $B$ is compact. \label{3p3_condii}
		\item\label{condition:regularity} $\{A,B\}$ is \emph{$\kappa$-linearly regular} on $\mathbb{B}(\overline{x},R)$ (the ball of radius $R$ about $\overline{x}$) for some $\overline{x} \in A \cap B$, where $R$ is big enough to ensure that $x_0 \in \mathbb{B}(\overline{x},R)$ and $\kappa>0$. That is, for all $x \in \mathbb{B}(\overline{x},R)$, $\dm(x,A\cap B) \leq  \kappa\max \left \{\dm(x,A),\dm(x,B) \right \}$. \label{3p3_condiii}
	\end{enumerate}
\end{theorem}
\begin{proof}
First we prove that the sequence is weakly convergent to $A\cap B$. Since Theorem \ref{thm:main} implies that  $\underset{n\rightarrow\infty}{\lim}\|x_n-\mathcal{P}_A(x_n)\| = \underset{n\rightarrow\infty}{\lim}\|x_n-\mathcal{P}_B \mathcal{R}^\gamma_A(x_n)\|= 0$, by assumptions \ref{Atozero} and \ref{Btozero}, we have that $\|x_{n}-P_{A}(x_{n})\|\rightarrow 0$ and $\|x_{n}-P_B(x_{n})\|\rightarrow 0$. Thus all weak cluster points of the sequence $(x_n)_{n\in\mathbb{N}}$ belong to $A$ and $B$, and so all weak cluster points of the sequence belong to $A\cap B$. By Theorem \ref{thm:main}, $\mathcal{T}^\lambda_{A^\gamma,B^\mu}$ is a quasinonexpansive operator, and $A\cap B \subseteq \Fix \mathcal{T}^\lambda_{A^\gamma,B^\mu}$, and so the sequence generated by \eqref{DRsequence} is Fej\'er monotone with respect to $A\cap B$. Since all weak cluster points belong to $A\cap B$ the whole sequence converges weakly to a point in $A\cap B$; see, for example \cite[Theorem 5.5]{Bau}. 
	\begin{enumerate}[label={(\roman*)}]
	\item This is obvious, since weak convergence implies strong in finite dimensional spaces.
	\item Suppose, without loss of generality, that $A$ is compact. Then, there exist a subsequence $(x_{k_n)_{k_n\in \mathbb{N}}}\subseteq (x_n)_{n\in\mathbb{N}}$ such that $\big(P_{A}(x_{k_n})\big)_{k_n\in \mathbb{N}}$ is strongly convergent to a point in $A$. Now, let $\bar{x}$ be the weak limit of the sequence $(x_n)_{n\in\mathbb{N}}$. Since $\bar{x}\in A\cap B$, we must have $P_{A}(x_{k_n})\rightarrow \bar{x}$. Now, 
	$$\|x_{k_n}-\bar{x}\|\leq \|x_{k_n}-P_{A}(x_{k_n})\|+\|P_{A}(x_{k_n})-\bar{x}\|\rightarrow 0,$$
	which proves that the sequence $(x_n)_{n\in\mathbb{N}}$ has a strong cluster point. By Theorem \ref{baucom} we conclude the strong convergence.
		
	\item Since $\dm(x_n, A\cap B)\leq \kappa \max \left \{\dm(x_n,A),\dm(x_n,B) \right \} \rightarrow 0$, using Theorem \ref{baucom}, we obtain the strong convergence. 
	\end{enumerate}
\end{proof}

Theorem~\ref{eg:phis} raises several natural questions. Firstly, it is evident that the conditions \ref{Atozero} and \ref{Btozero} are satisfied in the case of projections onto the constraint sets. We will give examples of other cutter methods which satisfy them in Section~\ref{sec:imp}. 

Next we show that even for a very simple setting of a singleton set $A$ it is possible to construct the constraint function in such a way that condition \ref{Atozero} does not hold, hence highlighting that this condition is essential for the result.

\begin{example}\label{eg:phis} Let $\HH = l_2$, and $A=\{0_{l_2}\}$. Note that $A$ is the zero level set of the function 
$$
f(x) = \sup_{k\in \NN} \varphi_k(x^{(k)}),
$$ 
where $x= (x^{(1)},x^{(2)},\dots, x^{(k)}, \dots)$ and 
$$
\varphi_k(t) = \max\left\{-\frac 1 k t, \frac 1 k t, k t + 1 - k, -k t +1 - k\right\}.
$$
These functions are shown in Figure~\ref{fig:phis} for $k\in \{1,\dots, 7\}$. 
\begin{figure}
{\centering \begin{overpic}[width=0.5\textwidth
	]{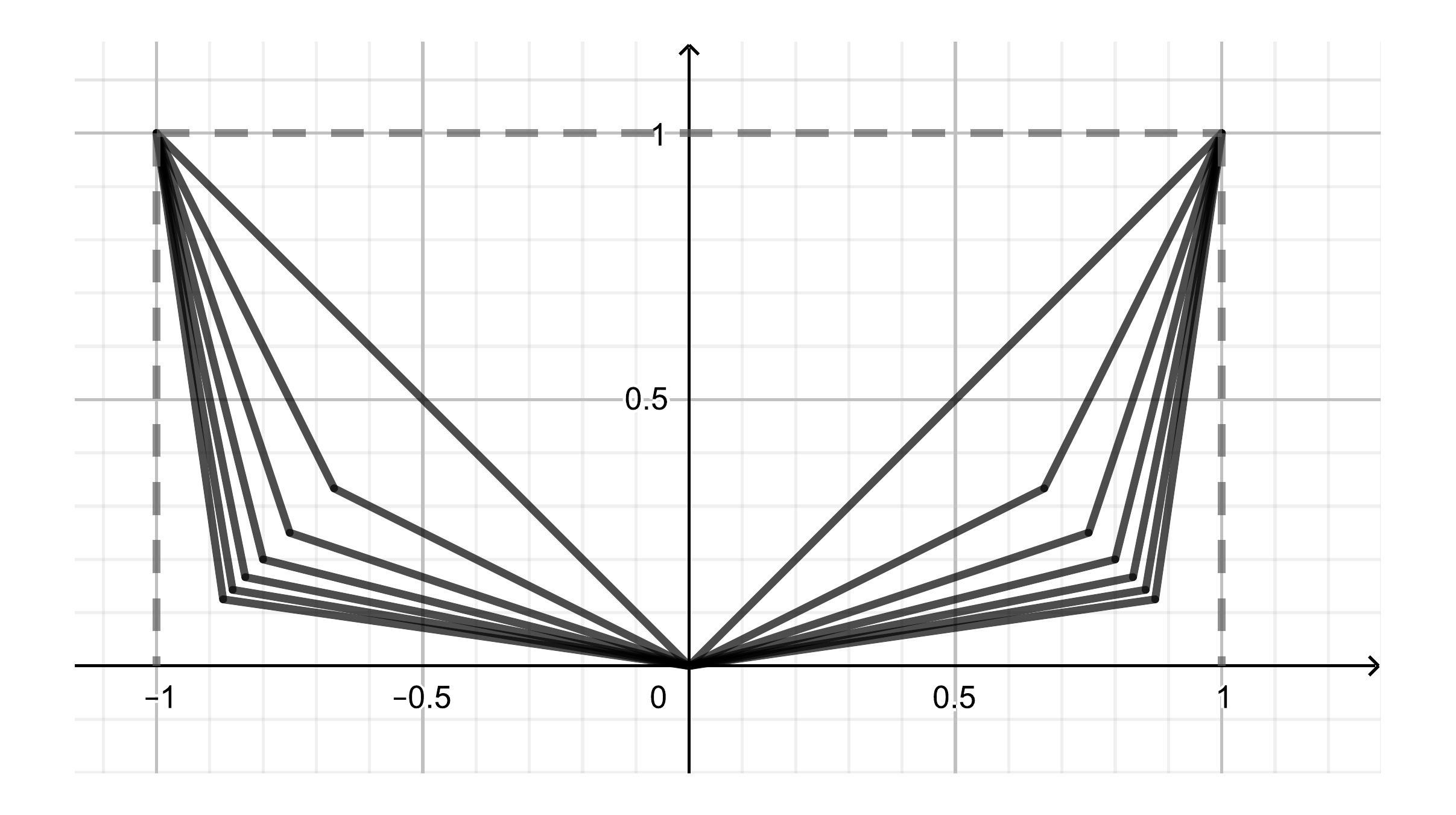}
	\put (30,30) {$\varphi_1$}
	\put (24,24) {$\varphi_2$}
	\put (60,30) {$\varphi_1$}
	\put (67,25) {$\varphi_2$}
\end{overpic}\\}	
\caption{The functions $\varphi_k$ from Example~\ref{eg:phis}.}
\label{fig:phis}
\end{figure}
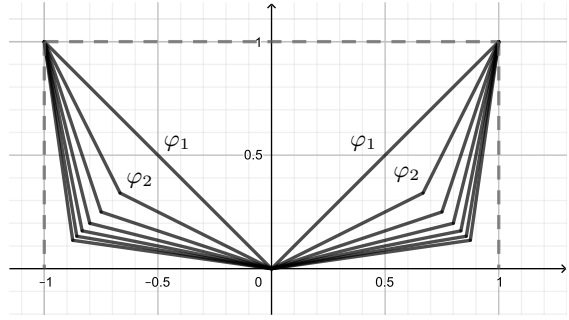
Indeed, for any nozero $x\in l_2$ we have $x^{(k)}\neq 0$ for at least one $k$, then $\varphi_k(x^{(k)})>0$, and hence $f(x)>0$. At the same time, for $x=0$ we have $\varphi_k(x^{(k)})=0$ for all  $k\in \NN$, so $f(0) = 0$.

Consider the sequence $\{x_n\}$, where $x_n$ has all entries zero except for $x_n^{(n)} = 1$, so we have 
$$ x_1  = (1,0,0,0,\dots),\quad x_2  = (0,1,0,0, \dots),\quad
x_3  = (0,0,1,0, \dots),\dots . 
$$
For $|t|\leq 1/2$ we have 
$$
|t|(k^2-1)-k^2+k \leq \frac{k^2-1}{2}-k^2+k = -\frac{k^2-2 k +1}{2} = - \frac{(k-1)^2}{2}\leq 0,
$$
hence
$$
\max\left\{ \frac 1 k |t|, k |t| + 1 - k\right\} 
= \frac {|t|} k  +\max\left\{0, \frac{|t|(k^2-1)-k^2+k}{k}\right\} = \frac {|t|} k,
$$
and 
$$
\varphi_k(t) = \frac 1 k |t| \quad \forall t, |t|\leq 1/2.
$$
At the same time, for $|t|\geq \frac{k}{k+1}$ we have 
$$
k |t| + 1 - k = \frac {|t|} k + \frac{k^2 -1}{k} |t| +1 - k \geq \frac {|t|} k,
$$
hence, 
$$
\max\left\{ \frac 1 k t, k t + 1 - k\right\} =  k t + 1 - k \quad \forall t,\,  |t|\geq \frac{k}{k+1},
$$
so we have for $\|u\|_{l_2}\leq \frac{1}{2n}$ that 
$$
\varphi_k(u^{(k)}) = \frac 1 k |u^{(k)}|\leq \frac 1 2 \quad \forall k \neq n, \quad 
\varphi_n(x^{(n)}+u^{(n)}) = \varphi_n(1+u^{(n)}) =  n u^{(n)} + 1\geq \frac 1 2,
$$  
hence  
$$
f(x_n+u)  = \max\{\varphi_n(x_n^{(n)}+u^{(n)}),\sup_{k\neq n}\varphi_k(u^{(k)})\} = \varphi_n(x_n^{(n)}+u^{(n)}),
$$
and so in a small neighbourhood of $x^{(n)}$ we have 
$$
f(x) = \varphi_n(x^{(n)})= n x^{(n)} + 1 - n.
$$
The subgradient cutter then gives 
$$
\|x_n - P_{\partial f}(x_n)\| = \frac{1}{n}\to 0, 
$$
however 
$$
d (x_n,A) = \|x_n\| = 1,
$$
so the condition \ref{Atozero} is violated.
\end{example}

Due to an important example by Hundal \cite{Hundal}, we know that in infinite dimensions our algorithms may fail if we don't have subtransversality or compactness. The above theorem also begs the question of what may go wrong in the case where we allow reflections $\gamma=0$ or $\mu=0$. 

\begin{example}\label{eg:fixedpoints1}
Letting $f,g:\mathbb{R} \rightarrow [0,\infty)$ by $f(x):=|x|=:g(x)$. Then every point in $\RR$ is a fixed point of $T_{f^{\gamma=0},g^{\mu=0}}$. This example is illustrated at left in Figure~\ref{fig:abs}.
\end{example}

For this example, all of the fixed points satisfy the property that $P_f(x) \in A\cap B$, which is analogous to the classical Douglas--Rachford fixed point result in Theorem~\ref{thm:LionsandMercier}. This property does not always hold, however, as illustrated in the next example.

\begin{example}\label{eg:fixedpoints2}Let $f,g:\mathbb{R}^2 \rightarrow [0,\infty)$ by $f,g:(x,y) \mapsto \max \{|x|,|y|\}$. This example is illustrated at right in Figure~\ref{fig:abs}. Any point $(x,y)$ satisfying $|x| \neq |y|$ is a fixed point of the operator $T_{f^{\gamma=0},g^{\mu=0}}$; indeed, it is possible that every point is a fixed point, depending upon how the cutter is chosen when $|x| = |y|$. If additionally, $x\neq 0$ and $y \neq 0$, then $(x,y)$ does not satisfy the property that $P_f(x) \in A\cap B$.
\end{example}

\begin{figure}[ht]
	\begin{center}
	\begin{subfigure}{.5\textwidth}
	\begin{tikzpicture}[scale=3.5]
	\draw [black] (-1,1) -- (0,0);
	\draw [black] (0,0) -- (1,1);
	
	\draw [fill,red] (1,0) circle [radius=.015];
	\node [below,red] at (1,0) {$x_0$};
	
	\draw [red,->] (1,.1) -- (1,.9);
	\draw [red,->] (1,.95) -- (.1,.05);
	\draw [fill,red] (0,0) circle [radius=.015];
	\node [below left,red] at (0,0) {$\mathcal{P}_A x_0=$};
	\node [below right,blue] at (0,0) {$\mathcal{P}_B \mathcal{R}_A^{\gamma=0} x_0$};
	
	\draw [blue,fill] (-1,0) circle [radius=.015];
	\node [below,blue] at (-1,0) {$\mathcal{R}_A^{\gamma=0} x_0$};
	\draw [red,->] (-.1,0) -- (-.9,0);
	
	\draw [blue,->] (-1,.1) -- (-1,.9);
	\draw [blue,->] (-1,.95) -- (-.1,.05);
	\draw [blue,->] (.1,0) -- (.9,0);
	
	\end{tikzpicture}
	\end{subfigure}
	\begin{subfigure}{.48\textwidth}
	\begin{adjustbox}{trim=0.3cm 0.5cm 0cm 0.8cm,clip=true}
	\begin{tikzpicture}[scale=1]
	\node[anchor=south west,inner sep=0] (image) at (0,0) {\includegraphics[width=.8\textwidth]{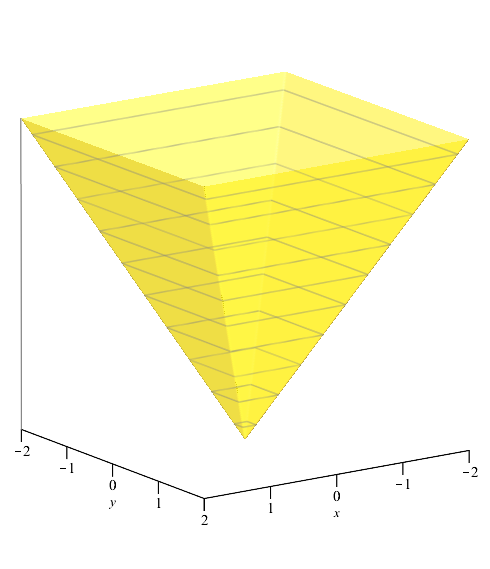}};
	\begin{scope}[x={(image.south east)},y={(image.north west)}]
	\draw [fill,red] (.375,.175) circle [radius=0.01];
	\node [above left,red] at (.375,.175) {$x_0$};
	
	\draw [fill,red] (.581,.205) circle [radius=0.01];
	\draw [red,->] (.375,.21) -- (.375,.6);
	\draw [red,->] (.396,.561) -- (.56,.245);
	
	\draw [fill,red] (.787,.235) circle [radius=0.01];
	\draw [red,->] (.622,.211) -- (.746,.229);
	
	\draw [blue] (.787,.275) -- (.787,.555);
	\draw [cyan,dashed,->] (.787,.555) -- (.787,.665);
	\draw [cyan,dashed] (.766,.619) -- (.684,.435);
	\draw [blue,->] (.684,.435) -- (.602,.251);
	
	\draw [blue,->] (.539,.199) -- (.416,.181);
	
	\end{scope}
	\end{tikzpicture}
	\end{adjustbox}
	\end{subfigure}
	
	\end{center}
	\caption{It is possible for every point to be a fixed point.}\label{fig:abs}
\end{figure}
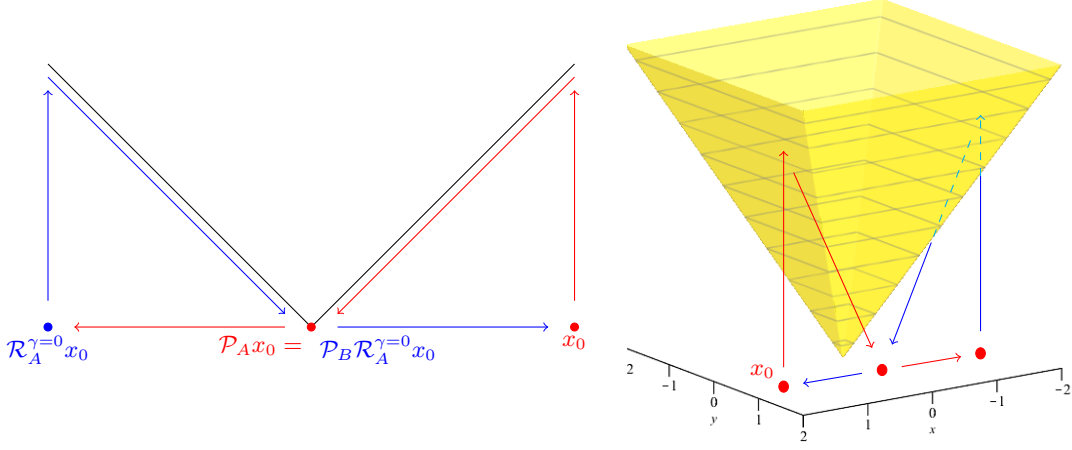

\begin{example}
	One might also ask of the regularity conditions of Theorem~\ref{thm:ABtozero}\ref{condition:regularity} can be used to guarantee linear convergence rates, as is often the case with projection operators (see the many convergence results listed in \cite{DRsurvey}). However, Theorem~\ref{eg:phis} is for \emph{very general} cutters, and so we can construct a counterexample. 
	
	Let $A:=B:=\{0\} \subset \RR$. It is straightforward to verify that $\{A,B\}$ is $1$-linearly regular on $\mathbb{B}(0,R)$ for any radius $R$. Let $C:=\left \{1/n \; | \; n \in \ZZ \setminus \{0\} \right \}$. Now define
	$$
	\mathcal{P}:x \mapsto \begin{cases}
	0 & \text{if}\; x=0\\
	1 & \text{if}\; x>1\\
	-1 & \text{if}\; x<1\\
	1/(n+1) & \text{if}\;\; 0< x = 1/n \in C\\
	1/(n-1) & \text{if}\;\; 0> x = 1/n \in C\\
	1/(n-1) \;\; \text{for the unique $n \in \mathbb{Z}$ satisfying}\;\; 1/n < x < 1/(n-1) & \text{if} \;\;-1< x < 0 \;\text{and}\; x \notin C\\
	1/(n+1) \;\; \text{for the unique $n \in \mathbb{Z}$ satisfying}\;\; 1/(n+1) < x < 1/n & \text{if}\;\; 1> x > 0\; \text{and}\; x \notin C    
	\end{cases}
	$$
	Clearly $\mathcal{P}$ is a cutter with respect to $A$ and $B$. Set  $\mathcal{P}_A:=\mathcal{P}_B:=\mathcal{P}$, $\gamma=\mu=1$, and $\lambda = 1$. Then for $x_0 := 1$,  we have $x_n:=1/(2n+1)$, so $x_n \rightarrow 0$ with a sublinear convergence rate.
\end{example}

\section{Implementations}\label{sec:imp}

For the classic implementation of the projection method with projections onto the constraint sets, the assumptions of Theorem~\ref{thm:ABtozero} are satisfied automatically, and hence we have the following result.

\begin{corollary}\label{trueprojection}
If $\mathcal{P}_A:=P_A$ and $\mathcal{P}_B:=P_B$ are projections onto the constraint sets, then assumptions \ref{Atozero} and \ref{Btozero} in Theorem~\ref{thm:ABtozero} are satisfied, and we have the same weak convergence results.
\end{corollary} 
\begin{proof}
Since in this case $\mathcal{P}_A=P_A$,  we have 
\begin{equation}\label{eq:I-cor}
\|x_n-\mathcal{P}_A(x_n)\| =\|x_n-P_A(x_n)\| = \dm(x_n,A)\xrightarrow[n\to \infty]{} 0,
\end{equation}
hence \ref{Atozero} holds. Additionally, since $P_B R_A^\gamma(x_n) \in B$,
\begin{align}\label{eq:bound-cor}
\dm(x_n,B) \leq \|x_n-P_BR^\gamma_A(x_n)\|\xrightarrow[n\to \infty]{} 0.
\end{align}
\end{proof}

Suppose that instead of two sets $A$ and $B$ we are given a finite collection of closed convex sets $\Omega_i\subseteq \HH$, where $i\in \{1,\dots,N\}$. The feasibility problem in this case consists of finding a point $x$ such that
\begin{equation}
 x\in \Omega:= \bigcap_{i=1}^{N} \Omega_i.\notag
\end{equation}
Our two set formulation can be applied to this setting by working in the product space $\mathcal{H}^N$, and letting
$$
A:=\Omega_1 \times \dots \times \Omega_N, \qquad B:= \{x=(u_1,\dots,u_N) \,|\, u_1=u_2=\dots = u_N \},
$$
in which case the product space projections for $x=(u_1,\dots,u_N)\in \HH^N$ are
\begin{align}  P_A(x)&=P_{\Omega_1,\times\dots\times\Omega_N}(x)=(P_{\Omega_1}(u_1),\dots, P_{\Omega_N}(u_N)).\nonumber \\
P_B(x)&=\left(\frac{1}{N}\sum_{k=1}^N u_k,\dots,\frac{1}{N}\sum_{k=1}^N u_k\right).\label{eqn:agreementprojection}
\end{align}
This well-known technique is used extensively in practical applications; see the important works of Pierra and Spingarn \cite{Pierra,spingarn1983partial}. We note that even in the elementary case of alternating or cyclic projection method the convergence is much easier to study and understand in the case of two sets, and in fact there are some negative results in terms of the shape of limit sets for the infeasible case of the problem on more than two sets \cite{cominetti,crw}.

We may use cutter methods together with the product space method to solve the \emph{system of inequalities} expressed in feasibility form as 
$$
x \in \bigcap_{i=1}^N \lev_{\leq 0}f_i
$$ 
For example, one may employ \emph{subgradient projections} with the cutter operators $P_{\partial f_i}$ defined by \eqref{eq:subgr}. From now on, we work in the Euclidean setting, letting $\mathbb{E}$ represent a finite dimensional Hilbert space. We first prove the convergence for the special case of two convex functions.

\begin{corollary}\label{subgradientpro}
Let $A:=\lev_{\leq 0}f$ and $B:=\lev_{\leq 0}g$ where $f:\mathbb{E} \rightarrow \RR$ and $g:\mathbb{E} \rightarrow \RR$ are convex functions with full domain. Suppose that $A\cap B \ne \emptyset$. 
Then the sequence $(x_n)_{n\in\NN}$ generated by  $x_{n+1} := T_{f^\gamma,g^\mu}^\lambda(x_n)$ with $\gamma,\mu \in (0,2)$ converges strongly to a point $\bar{x}\in A\cap B$.
\end{corollary}
\begin{proof}
By Theorem \ref{thm:main}\ref{maintheorem1}, we have that $(x_n)_{n\in\NN}$ is Fej\'er monotone with respect to $A\cap B$. Thus $(x_n)_{n\in\NN}$ is bounded. 

First we will prove that the conditions \ref{Atozero} and \ref{Btozero} in Theorem~\ref{thm:ABtozero} are satisfied. That is, $\|x_n-P_{\partial f}(x_n)\|\rightarrow 0$ implies that $\dm(x_n,A)\rightarrow 0$, and $\|x_n-{P}_{\partial g} {R}_{\partial f}^\gamma(x_n)\|\rightarrow 0$ implies that $\dm(x_n,B)\rightarrow 0$. Note that
$$\|x_n -P_{\partial f}(x_n)\|=\Big\|x_n-\Big(x_n-\frac{f(x_n)}{\|s(x_n)\|^2}s(x_n)\Big)\Big\|=\frac{|f(x_n)|}{\|s(x_n)\|},$$
where $s(x_n) \in \partial f(x_n)$. Since the sequence $(x_n)_{n\in\NN}$ is bounded, and $f$ has full domain, we have that the sequence $\|s(x_n)\|_{n\in\NN}$ is bounded (see, for example, \cite[Theorem 24.7]{Roc70}). Since  $\frac{f(x_n)}{\|s(x_n)\|}\rightarrow 0$, we have that $f(x_n)\rightarrow 0$. 

 We will show that $f(x_n)\rightarrow 0 \implies \dm(x_n,A) \rightarrow 0$. Let 
 $$
 D:=\mathbb{B}(P_{A \cap B}(x_0),\|x_0-P_{A \cap B}(x_0)\|).
 $$
 If $\|x_0-P_{A \cap B}(x_0)\| = 0$ then we are done. Suppose then that $\|x_0-P_{A \cap B}(x_0)\| >0$. Let $\|x_0-P_{A \cap B}(x_0)\| >\epsilon >0$. 
 
 Since $x_n$ is Fej\'er monotone with respect to $A\cap B$, we have that $x_n \in D$ for all $n$. Thus we may work with a restriction of $f$:
 $$
 f|_{D}: x \mapsto \begin{cases}
 f(x) & \text{if} \;\; x \in D;\\
 \infty & \text{otherwise},
 \end{cases}
 $$
 which is convex and coercive and satisfies $f|_{D}(x_n)=f(x_n)$ for all $n$, as well as $A':=A \cap D = \lev_{\leq 0} f|_D$.
 
 Without loss of generality let $0 \in A'$. Let $S:= A'+\mathbb{B}(0,\epsilon)$. As $A'$ is bounded, $S$ is bounded. The condition $\|x_0-P_{A \cap B}(x_0)\| >\epsilon$ ensures that $\bdry S \cap D \neq \emptyset$. As $f|_{D}$ is proper, continuous, and $\bdry S$ is closed and bounded with $\bdry S \cap D \neq \emptyset$, $f|_{D}$ attains a minimum on $\bdry S$. Let $\zeta := {\rm min}_{x \in \bdry S}{f|_D(x)}$. We will show $\lev_{\leq \zeta}f|_{D} \subset S$. Suppose by way of contradiction that there exists $y \in \lev_{\leq \zeta}f|_{D} \setminus S$. Then since $\mathbb{B}(0,\epsilon) \subset S$, we have that $y = \frac{1}{\lambda} u$ for some $u \in \bdry S$ and $\lambda \in [0,1]$. Thus we have $u = \lambda y + (1-\lambda)0$. This yields
 \begin{equation}\label{proving4p2}\zeta \leq f|_{D}(u) = f|_{D}(\lambda y + (1-\lambda)0) \leq \lambda f|_{D}(y)+(1-\lambda)f|_{D}(0) \leq \lambda f|_{D}(y) \leq \lambda \zeta\end{equation} 
 where the first inequality is how we have defined $\zeta$, the first equality is from how we have defined $u$, the second inequality is from convexity of $f|_D$, the third is because $0 \in A'=\lev_{\leq 0}f|_{D}$, and the final inequality is because $0<f|_{D}(y) \leq \zeta$. From \eqref{proving4p2}, we have $\zeta \leq \lambda \zeta$, which is true only if $\lambda =1$ or $\zeta = 0$. If $\zeta = 0$, then $y \in \lev_{\leq 0}f|_D \subset S$, a contradiction. If $\lambda = 1$ then $y=u \in S$, a contradiction. Thus $\lev_{\leq \zeta}f|_D \subset S$. Since $f|_D(x_n) \rightarrow 0$, there exists $N \in \NN$ such that for all $n \geq N$, $f|_D(x_n) < \zeta$ and so $x_n \in \lev_{\leq \zeta} f|_D \subset S$ and so $\dm(A,x_n) \leq \dm(A',x_n) \leq \epsilon$. Thus $\dm(A,x_n) \rightarrow 0$.
 
 Turning to the function $g$, with the same argument that is in Corollary \ref{trueprojection} we have
 $$\|x_n-P_{\partial g}(x_n)\|\leq \|x_n-P_{\partial g}R_{\partial f}^\gamma(x_n)\|+(2-\gamma)\|P_{\partial f}(x_n)-x_n\| \rightarrow 0.$$
 Thus, by the same arguments we used to show $f(x_n) \rightarrow 0$, we have that  $g(x_n)\rightarrow 0$ and that $\dm(x_n,B)\rightarrow 0$.
 
 Together with the fact that $x_n$ is Fej\'{e}r monotone with respect to $A \cap B$ and the fact that $\mathbb{E}$ satisfies condition \ref{3p3_condi} from Theorem~\ref{thm:ABtozero}, we conclude by Theorem \ref{baucom} the convergence of the sequence for a point in $A\cap B$. 
\end{proof}

Now we present a result for the case of more than two functions. 
\begin{corollary}
Let the $(f_i)_{i\in \mathcal{I}}$ where $\mathcal{I}=\{1,2,\cdots,N\}$, $N\in\NN$, are convex functions from $\mathbb{E}$ to $\RR$. Consider for all $i\in\mathcal{I}$, the sets $A_i:=\{x\in \mathbb{E}: f_i(x)\leq 0\}$, and suppose that $C:=\cap_{i\in\mathcal{I}}A_i\neq \emptyset$. Consider the functions 
$$F: \mathbb{E}^N\rightarrow \RR \ \ \textup{ defined by: } (x_1,x_2,\cdots,x_N)\to \sum_{i\in\mathcal{I}}\max\{f_i(x_i),0\},$$
$$G:\mathbb{E}^N\rightarrow \RR \ \ \textup{ defined by: } (x_1,x_2,\cdots,x_N)\to \sum_{i\in\mathcal{I}}\Big\|x_i-\frac{1}{N}\sum_{j\in\mathcal{I}} x_j \Big\|^2.$$
Let the sequence $({\bf x}_n)_{n\in\NN}$ be as follows:
\begin{eqnarray*}
{\bf x}_0&=&(x_1^0,x_2^0,\cdots, x_N^0)\in \mathbb{E}^N,\\
{\bf x}_{n+1}&=&(x_1^{n+1},x_2^{n+1},\cdots, x_N^{n+1})=\mathcal{T}^\lambda_{F^\gamma,G^\mu}({\bf x}_n)=\mathcal{T}^\lambda_{F^\gamma,G^\mu}(x_1^n,x_2^n,\cdots,x_N^n).
\end{eqnarray*} 
 Then ${\bf x}_n \rightarrow {\bf \bar{x}}=(\bar{x},\bar{x},\cdots,\bar{x})\in D:= \Pi_{i\in \mathcal{I}}A_i $ with $\bar{x}\in \mathbb{E}$, which means that $\bar{x}\in C$.
\end{corollary}
\begin{proof}
The convexity of each $f_i$ guarantees the convexity of $F$. Notice that $B:=\lev_{\leq 0}G=\{(x_1,x_2,\cdots,x_N)\in \mathbb{E}^N : x_1=x_2=\cdots =x_N\}$ is the linear subspace of agreement which we recognize from \eqref{eqn:agreementprojection}, and $G={\rm d}(B,\cdot)^2$ is actually the square of the distance function for $B$. As $G$ is the square of the distance function for a convex set, $G$ is convex. In fact, if one chooses to replace subgradient projection with respect to $G$ by Euclidean projection directly onto its zero level set, the Euclidean projection just as given in \eqref{eqn:agreementprojection}.

The algorithm is well defined because the domain of each $f_i$ is the space $\mathbb{E}$. Finally notice that $D = \lev_{\leq 0}F$. Applying Corollary \ref{subgradientpro} we have that ${\bf x}_n\rightarrow {\bf \bar{x}}:=(\bar{x},\bar{x}, \cdots, \bar{x})\in D\cap B$ where $B=\{(x_1,x_2,\cdots,x_N)\in \mathbb{E}^N : x_1=x_2=\cdots =x_N\}$.
\end{proof}
 
As an immediate consequence, we also have strong convergence in the case where we work with projections onto the constraint sets and a finite number of sets $A_1,\dots,A_N$; just let the $N$ functions be given by $f_i:=\dm_{A_i}(\cdot)$. See, for example, \cite[Ex.~2.7]{BWWX3}.

\begin{remark}[Sequences $\gamma_n,\mu_n$]
	One may take sequences $(\gamma_n)_{n \in \NN},(\mu_n)_{n \in \NN}$ and, provided that $\liminf \gamma_n(2-\gamma_n) >0$ and $\liminf \mu_n(2-\mu_n) >0$, all of the above convergence results will hold for sequence given by $x_n := \mathcal{T}_{A^{\gamma_n},B^{\mu_n}}^\lambda x_{n-1}$. Indeed, this is the usual framework of \cite{Cegielski}, although we have avoided the use of these sequences for the simplicity of exposition.
\end{remark}
 
\begin{figure}
	\begin{center}
		\begin{tikzpicture}[scale=.5]
		\node[anchor=south west,inner sep=0] (image) at (0,0) {\includegraphics[angle=90,width=.6\textwidth]{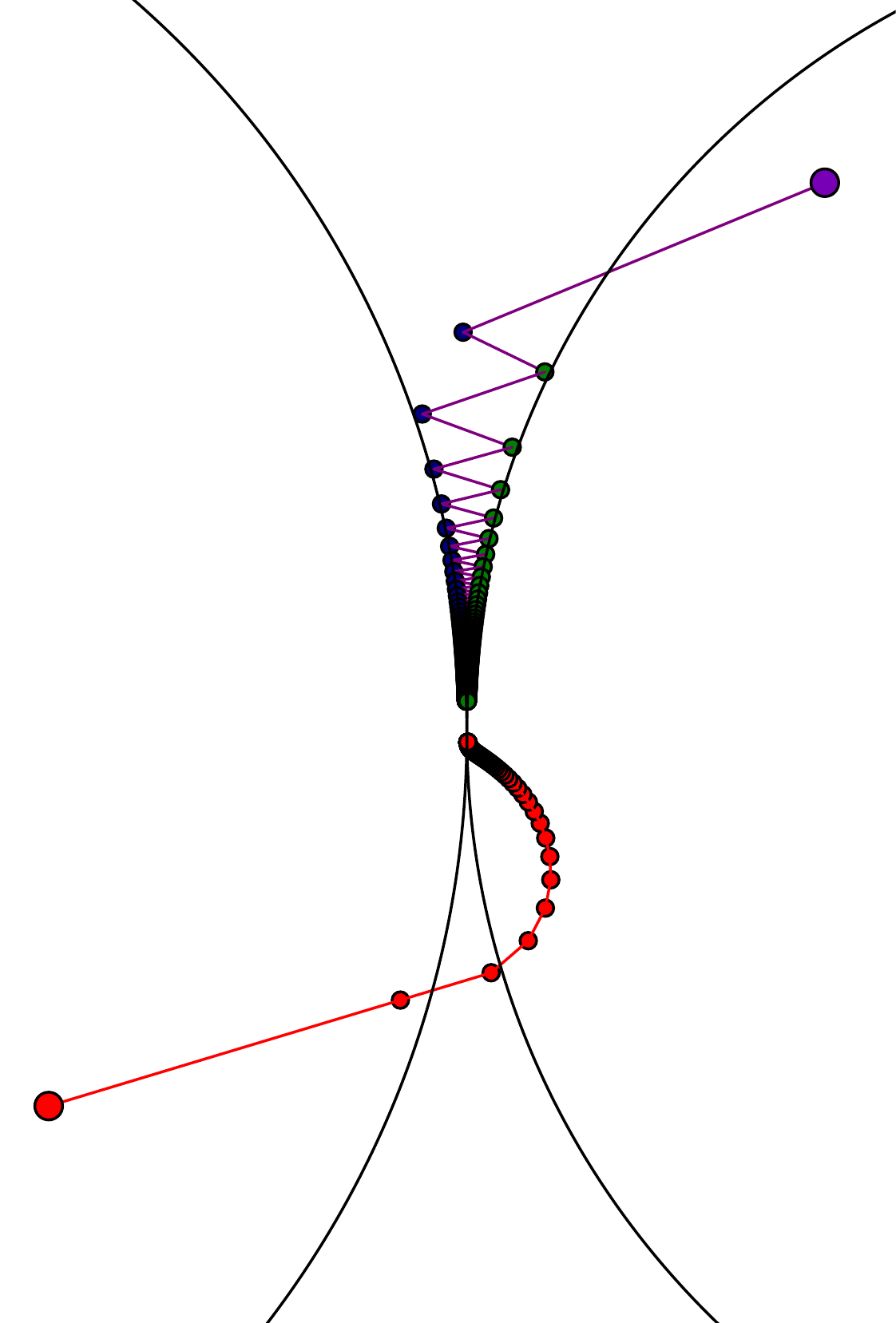}};
		\begin{scope}[x={(image.south east)},y={(image.north west)}]
		
		\node [right,purple] at (.15,.9) {$\mathcal{T}_{f^{\gamma=1},g^{\mu=1}}^{\lambda=1}$};
		
		\node [] at (.5,.8) {$\lev_{\leq 0}g$};
		
		\node [] at (.5,.3) {$\lev_{\leq 0}f$};
		
		\node [left,red] at (.8,.1) {$\mathcal{T}_{f^{\gamma=1/10},g^{\mu=1/10}}^{\lambda=\frac{1}{2}}$};
		
		\end{scope}
		\end{tikzpicture}
	\end{center}
	\caption{Convergence for $\mathcal{T}_{f^{\gamma=1/10},g^{\mu=1/10}}^{\lambda=\frac{1}{2}}$ (parameters similar to Douglas--Rachford method) vs. $\mathcal{T}_{f^{\gamma=1},g^{\mu=1}}^{\lambda=1}$ (parameters similar to alternating projections)}
	\label{fig:DRvsAP}
\end{figure}
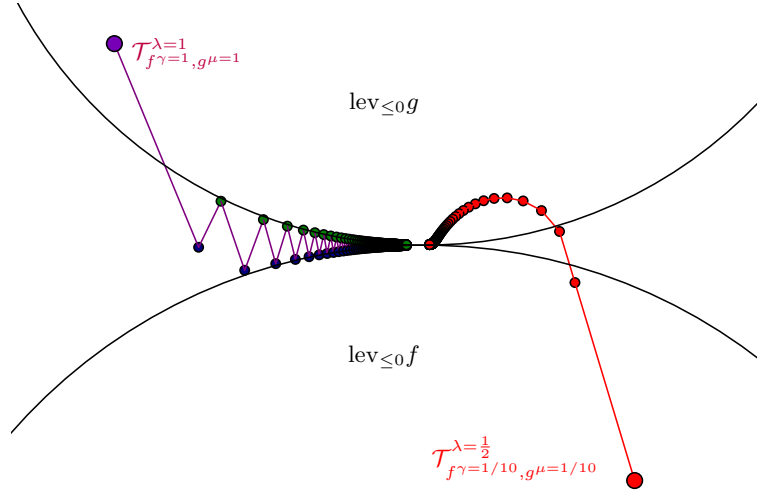
 
\section{Discussion}

In the convex setting, when projections onto the constraint sets are replaced with cutters, the operator $\mathcal{T}_{A^\gamma,B^\mu}^\lambda$ loses firm nonexpansivity and yet retains many of its desirable convergence properties because of Fej\'{e}r monotonicity. Subgradient projections are one useful context in which the firm nonexpansivity is lost while the Fej\'{e}r monotonicity is retained. The similarities suggest several avenues of further research: one in the convex setting and one outside of it. 

\subsection{Further Investigation}

In the convex setting, the algorithmic differences corresponding to different choices of $\mu,\gamma,\lambda$ are a highly active area of investigation. See, for example, \cite{artacho2017optimal}, \cite{dao2016linear}, and \cite{dao2017generalized}. \todo{add references} Figure~\ref{fig:DRvsAP} compares two variants of $\lambda$-averaged relaxed projection methods in the case of subgradient projections, and the behaviour differences are reminiscent of those known in the setting of projections onto the constraint sets. Further comparison of behaviour for choices of averaging and relaxation parameters invites experimental investigation.

Even when the formulation of a problem allows for computations of projections onto the constraint sets, it may be undesirable (computationally expensive) to do so. Consider, for example, the projection onto an ellipse: $\mathcal{E}:= \left \{(x,y) \in \RR^2 \;|\; (x-a)^2+(y-b)^2 = 1 \right \}$ for given constants $a,b$. Computation of the exact projection for a point not in $\mathcal{E}$ requires solving a Lagrangian problem (see, for example \cite{BLSSS} and \cite{LSS}), while computation of the subgradient projection for the function $f:(x,y) \rightarrow ((x-a)^2+(y-b)^2 -1)^2$ does not. It is very natural to investigate the differences in behaviour induced by the choice of projection method.

Both the method of alternating projections and the Douglas--Rachford method have also been used to solve a variety of nonconvex feasibility problems, with the latter generally the more robust. See, for example, \cite{AB,ABT,FranMattJonGlobal,DRMxcomp,AC,Benoist,BLSSS,lamichhane2017application}, and \cite{LSS}. It is reasonable to consider the behaviour of $\lambda$-averaged relaxed projection methods in the nonconvex setting, and a very natural problem would be that of finding $x \in \lev_{\leq 0}f \cap \lev_{\leq 0}g$---using subgradient projections---where one or both of $f,g$ are not convex. Indeed, any nonconvex feasibility problem in $\RR^N$ is an example of such a nonconvex variational inequality problem where $f=\dm_A(\cdot),g=\dm_B(\cdot)$, and so much investigation has already been done. 

\todo{Vera: Mention metric regularity and results for semidefinite sets by Guoyin Li and co (Scott: I'll let you add this since I'm less familiar with it than you are)}

\subsection{Conclusion}

We learn much by comparing and contrasting what may be shown about $\lambda$-averaged relaxed generalized projection methods through the differing frameworks of firm nonexpansivity and quasi-nonexpansivity. That so many of the desirable properties carry over---from the more specific setting of projections onto the constraint sets to the more general setting of cutters---is especially useful. Splitting methods employing projections onto the constraint sets are an area of significant experimental research. We conclude by noting that those methods which employ other implementations of cutters merit further experimental investigation, and that the theory is elegant in its own regard.

\textbf{Acknowledgements}

We are grateful to Heinz Bauschke for pointing out useful references. We also extend our gratitude to the organizers and participants of the workshop \emph{Splitting Algorithms, Modern Operator Theory and Applications} held at Casa Matematica Oaxaca, Mexico, from 17--22 September 2017 and to the Australian Research Council for supporting our travel expenses via project DE150100240.

We dedicate this work to the memory of Jonathan M. Borwein, whose influence on the present authors and on the Australian mathematical community cannot be overstated.

\bibliographystyle{plain}
\bibliography{refs}

\begin{thebibliography}{10}

\bibitem{reic}
Arkady Aleyner and Simeon Reich.
\newblock Block-iterative algorithms for solving convex feasibility problems in
  {H}ilbert and in {B}anach spaces.
\newblock {\em J. Math. Anal. Appl.}, 343(1):427--435, 2008.

\bibitem{AB}
Francisco~J. Arag\'{o}n~Artacho and Jonathan~M. Borwein.
\newblock Global convergence of a non-convex {D}ouglas--{R}achford iteration.
\newblock {\em J. Glob. Optim.}, 57(3):753--769.

\bibitem{DRMxcomp}
Francisco~J. Arag\'on~Artacho, Jonathan~M. Borwein, and Matthew~K. Tam.
\newblock {D}ouglas--{R}achford feasibility methods for matrix completion
  problems.
\newblock {\em ANZIAM J.}, 55(4):299--326, 2014.

\bibitem{ABT}
Francisco~J. Arag\'on~Artacho, Jonathan~M. Borwein, and Matthew~K. Tam.
\newblock Recent results on {D}ouglas--{R}achford methods for combinatorial
  optimization problems.
\newblock {\em J. Optim. Theory Appl.}, 163(1):1--30, 2014.

\bibitem{FranMattJonGlobal}
Francisco~J. Arag\'on~Artacho, Jonathan~M. Borwein, and Matthew~K. Tam.
\newblock Global behavior of the {D}ouglas--{R}achford method for a nonconvex
  feasibility problem.
\newblock {\em J. Global Optim.}, 65(2):309--327, 2016.

\bibitem{AC}
Francisco~J. Arag\'on~Artacho and Rub\'en Campoy.
\newblock Solving graph coloring problems with the {D}ouglas--{R}achford
  algorithm.
\newblock {\em Set-Valued and Variational Analysis}, pages 1--28, 2017.

\bibitem{FranNewMethod}
Francisco~J. Arag\'on~Artacho and Rub\'en Campoy.
\newblock A new projection method for finding the closest point in the
  intersection of convex sets.
\newblock {\em Comput. Optim. Appl.}, 69(1):99--132, 2018.

\bibitem{artacho2017optimal}
Francisco J.~Arag{\'o}n Artacho and Rub{\'e}n Campoy.
\newblock Optimal rates of linear convergence of the averaged alternating
  modified reflections method for two subspaces.
\newblock {\em arXiv preprint arXiv:1711.06521}, 2017.

\bibitem{cominetti}
J.-B. Baillon, P.~L. Combettes, and R.~Cominetti.
\newblock There is no variational characterization of the cycles in the method
  of periodic projections.
\newblock {\em J. Funct. Anal.}, 262(1):400--408, 2012.

\bibitem{BC}
Heinz~H. Bauschke and Patrick~L. Combettes.
\newblock {\em Convex analysis and monotone operator theory in {H}ilbert
  spaces}.
\newblock CMS Books in Mathematics/Ouvrages de Math\'ematiques de la SMC.
  Springer, Cham, second edition, 2017.
\newblock With a foreword by H\'edy Attouch.

\bibitem{BCL}
Heinz~H. Bauschke, Patrick~L. Combettes, and D.~Russell Luke.
\newblock Phase retrieval, error reduction algorithm, and {F}ienup variants: a
  view from convex optimization.
\newblock {\em J. Opt. Soc. Amer. A}, 19(7):1334--1345, 2002.

\bibitem{Bau}
Heinz~H. Bauschke, Minh~N. Dao, Dominikus Noll, and Hung~M. Phan.
\newblock On {S}later's condition and finite convergence of the
  {D}ouglas--{R}achford algorithm for solving convex feasibility problems in
  {E}uclidean spaces.
\newblock {\em J. Global Optim.}, 65(2):329--349, 2016.

\bibitem{BWWX2}
Heinz~H Bauschke, Caifang Wang, Xianfu Wang, and Jia Xu.
\newblock On subgradient projectors.
\newblock {\em SIAM Journal on Optimization}, 25(2):1064--1082, 2015.

\bibitem{BWWX1}
Heinz~H Bauschke, Caifang Wang, Xianfu Wang, and Jia Xu.
\newblock On the finite convergence of a projected cutter method.
\newblock {\em Journal of Optimization Theory and Applications},
  165(3):901--916, 2015.

\bibitem{BWWX3}
Heinz~H Bauschke, Caifang Wang, Xianfu Wang, and Jia Xu.
\newblock Subgradient projectors: Extensions, theory, and characterizations.
\newblock {\em Set-Valued and Variational Analysis}, pages 1--70, 2017.

\bibitem{RelaxedProjection}
J.~Y. Bello~Cruz and R.~D\'iaz~Mill\'an.
\newblock A relaxed-projection splitting algorithm for variational inequalities
  in {H}ilbert spaces.
\newblock {\em J. Global Optim.}, 65(3):597--614, 2016.

\bibitem{BelloCruzIusem}
J.~Y. Bello~Cruz and A.~N. Iusem.
\newblock An explicit algorithm for monotone variational inequalities.
\newblock {\em Optimization}, 61(7):855--871, 2012.

\bibitem{Benoist}
Joel Benoist.
\newblock The {D}ouglas--{R}achford algorithm for the case of the sphere and
  the line.
\newblock {\em J. Glob. Optim.}, 63:363--380, 2015.

\bibitem{hhm}
Jonathan~M. Borwein.
\newblock The life of modern homo habilis mathematicus: Experimental
  computation and visual theorems.
\newblock In {\em Tools and Mathematics}, volume 347 of {\em Mathematics
  Education Library}, pages 23--90. Springer, 2016.

\bibitem{BLT2015}
Jonathan~M. Borwein, Guoyin Li, and Matthew~K. Tam.
\newblock Convergence rate analysis for averaged fixed point iterations in
  common fixed point problems.
\newblock {\em SIAM J. Optim.}, 27(1):1--33.

\bibitem{AltSemialg}
Jonathan~M. Borwein, Guoyin Li, and Liangjin Yao.
\newblock Analysis of the convergence rate for the cyclic projection algorithm
  applied to basic semialgebraic convex sets.
\newblock {\em SIAM J. Optim.}, 24(1):498--527, 2014.

\bibitem{BLSSS}
Jonathan~M. Borwein, Scott~B. Lindstrom, Brailey Sims, Anna Schneider, and
  Matthew~P. Skerritt.
\newblock Dynamics of the {D}ouglas--{R}achford method for ellipses and
  p-spheres.
\newblock {\em Set-Valued and Variational Analysis}, 26(2):385--403, 2018.

\bibitem{JonBrailey}
Jonathan~M. Borwein and Brailey Sims.
\newblock The {D}ouglas--{R}achford algorithm in the absence of convexity.
\newblock In {\em Fixed-point algorithms for inverse problems in science and
  engineering}, volume~49 of {\em Springer Optim. Appl.}, pages 93--109.
  Springer, New York, 2011.

\bibitem{Cegielski}
Andrzej Cegielski.
\newblock {\em Iterative methods for fixed point problems in {H}ilbert spaces},
  volume 2057 of {\em Lecture Notes in Mathematics}.
\newblock Springer, Heidelberg, 2012.

\bibitem{CRZ}
Andrzej Cegielski, Simeon Reich, and Rafa{\l} Zalas.
\newblock Regular sequences of quasi-nonexpansive operators and their
  applications, 2017.

\bibitem{censor1982cyclic}
Yair Censor and Arnold Lent.
\newblock Cyclic subgradient projections.
\newblock {\em Mathematical Programming}, 24(1):233--235, 1982.

\bibitem{censor1997parallel}
Yair Censor and Stavros~Andrea Zenios.
\newblock {\em Parallel optimization: Theory, algorithms, and applications}.
\newblock Oxford University Press on Demand, 1997.

\bibitem{combettes1997convex}
Patrick~L Combettes.
\newblock Convex set theoretic image recovery by extrapolated iterations of
  parallel subgradient projections.
\newblock {\em IEEE Transactions on Image Processing}, 6(4):493--506, 1997.

\bibitem{crw}
Roberto Cominetti, Vera Roshchina, and Andrew Williamson.
\newblock A counterexample to {D}e {P}ierro's conjecture on the convergence of
  under-relaxed cyclic projections, 2018.

\bibitem{dao2016linear}
Minh~N. Dao and Hung~M. Phan.
\newblock Linear convergence of projection algorithms.
\newblock {\em arXiv preprint arXiv:1609.00341}, 2016.

\bibitem{dao2017generalized}
Minh~N. Dao and Hung~M. Phan.
\newblock Linear convergence of the generalized {D}ouglas--{R}achford algorithm
  for feasibility problems.
\newblock {\em arXiv preprint arXiv:1710.09814}, 2017.

\bibitem{DePierro}
Alvaro~R. De~Pierro.
\newblock From parallel to sequential projection methods and vice versa in
  convex feasibility: results and conjectures.
\newblock In {\em Inherently parallel algorithms in feasibility and
  optimization and their applications ({H}aifa, 2000)}, volume~8 of {\em Stud.
  Comput. Math.}, pages 187--201. North-Holland, Amsterdam, 2001.

\bibitem{DR}
Jim Douglas, Jr. and H.~H. Rachford, Jr.
\newblock On the numerical solution of heat conduction problems in two and
  three space variables.
\newblock {\em Trans. Amer. Math. Soc.}, 82:421--439, 1956.

\bibitem{DLW}
Dmitriy Drusvyatskiy, Guoyin Li, and Henry Wolkowicz.
\newblock A note on alternating projections for ill-posed semidefinite
  feasibility problems.
\newblock {\em Math. Program.}, 162(1-2, Ser. A):537--548, 2017.

\bibitem{elser2017matrix}
Veit Elser.
\newblock Matrix product constraints by projection methods.
\newblock {\em Journal of Global Optimization}, 68(2):329--355, 2017.

\bibitem{Fukushima}
Masao Fukushima.
\newblock A relaxed projection method for variational inequalities.
\newblock {\em Math. Programming}, 35(1):58--70, 1986.

\bibitem{Hundal}
Hein~S. Hundal.
\newblock An alternating projection that does not converge in norm.
\newblock {\em Nonlinear Analysis: Theory, Methods \& Applications}, 57(1):35
  -- 61, 2004.

\bibitem{subt}
Alexander~Y. Kruger, D.~Russell Luke, and Nguyen~H. Thao.
\newblock About subtransversality of collections of sets.
\newblock {\em Set-Valued Var. Anal.}, 25(4):701--729, 2017.

\bibitem{lamichhane2017application}
Bishnu~P. Lamichhane, Scott~B. Lindstrom, and Brailey Sims.
\newblock Application of projection algorithms to differential equations:
  boundary value problems.
\newblock {\em arXiv preprint arXiv:1705.11032}, 2017.

\bibitem{LP}
Guoyin Li and Ting~Kei Pong.
\newblock {D}ouglas-{R}achford splitting for nonconvex optimization with
  application to nonconvex feasibility problems.
\newblock {\em Math. Program.}, 159(1-2, Ser. A):371--401, 2016.

\bibitem{DRsurvey}
Scott~B. Lindstrom and Brailey Sims.
\newblock Survey: Sixty years of {D}ouglas--{R}achford.
\newblock {\em available at \url{https://arxiv.org/abs/1809.07181}}, 2018.

\bibitem{LSS}
Scott~B. Lindstrom, Brailey Sims, and Matthew~P. Skerritt.
\newblock Computing intersections of implicitly specified plane curves.
\newblock {\em Nonlinear and Conv. Anal.}, 18(3):347--359, 2017.

\bibitem{LM}
P.-L. Lions and B.~Mercier.
\newblock Splitting algorithms for the sum of two nonlinear operators.
\newblock {\em SIAM J. Numer. Anal.}, 16(6):964--979, 1979.

\bibitem{Littlewood}
John~Edensor Littlewood.
\newblock {\em A Mathematician's Miscellany}.
\newblock Methuen London, 1953.

\bibitem{Pierra}
Guy Pierra.
\newblock Decomposition through formalization in a product space.
\newblock {\em Mathematical Programming}, 28(1):96--115, 1984.

\bibitem{polyak1987introduction}
Boris~T. Polyak.
\newblock Introduction to optimization. {T}ranslations series in mathematics
  and engineering.
\newblock {\em Optimization Software}, 1987.

\bibitem{Roc70}
Ralph~Tyrell Rockafellar.
\newblock Convex analysis.
\newblock {\em Princeton University Press}, 1970.

\bibitem{spingarn1983partial}
Jonathan~E. Spingarn.
\newblock Partial inverse of a monotone operator.
\newblock {\em Applied mathematics and optimization}, 10(1):247--265, 1983.

\bibitem{Svaiter}
Benar~F. Svaiter.
\newblock On weak convergence of the {D}ouglas--{R}achford method.
\newblock {\em SIAM J. on Control and Opt.}, 49(1):280--287, 2011.

\end{thebibliography}

\listoftodos[Summary of to-do notes]

\end{document}